\documentclass[reqno]{amsart}

\usepackage[title]{appendix}
\usepackage[french,english]{babel}
\usepackage[utf8]{inputenc}
\usepackage[T1]{fontenc}
\usepackage{comment}
\usepackage{palatino}
\usepackage[autostyle]{csquotes}
\usepackage{amsmath}
\usepackage{amsfonts,amssymb,amsmath,latexsym,mathrsfs,bbm,stmaryrd}
\usepackage[all]{xy}
\usepackage[usenames]{color}
\usepackage{epsfig}
\usepackage{graphicx}
\usepackage{subcaption}
\usepackage{float,enumerate}
\usepackage{amsthm}
\usepackage{thmtools}
\usepackage{thm-restate}
\usepackage[mathcal]{euscript}
\usepackage{times}
\usepackage{xcolor}
\usepackage[export]{adjustbox}
\usepackage{hyperref}

\theoremstyle{plain}
\newtheorem{theorem}{Theorem}[section]

\newtheorem{assumption}[theorem]{Assumption}

\newtheorem{proposition}[theorem]{Proposition}

\newtheorem{lemma}[theorem]{Lemma}
\newtheorem{corollary}[theorem]{Corollary}

\theoremstyle{definition}
\newtheorem{definition}[theorem]{Definition}

\newtheorem{remark}[theorem]{Remark}
\newtheorem{notation}[theorem]{Notation}

\numberwithin{equation}{section} 

\newtheorem{assumpy}{Assumption}

\newcounter{tmp}
\usepackage[T1]{fontenc}

\renewcommand{\Im}{\operatorname{Im}}

\renewcommand\Im{\operatorname{Im}}

\def\R{{\mathbb R}}

\def\<{{\langle}}
\def\>{{\rangle}}

\begin{document}

\title{Well-posedness and stationary distribution of free stochastic differential equations}

\author{
Jiaxin Wei and Zhi Yin
}

\address{
\parbox{\linewidth}{Jiaxin Wei,
School of Mathematics and Statistics, 
Central South University,\\
Changsha, Hunan 410083, China.\\}
}

\address{
\parbox{\linewidth}{Zhi Yin,
School of Mathematics and Statistics, 
Central South University,\\
Changsha, Hunan 410083, China.\\}
}

\date{\today}
\maketitle

\begin{abstract}
This paper studies free stochastic differential equations driven by free Brownian motion. Under local operator Lipschitz and Lyapunov-type conditions on the coefficients, we prove the global well-posedness of solutions in the noncommutative probability setting using free It\^o calculus. We further establish the existence and uniqueness of stationary solutions under appropriate dissipativity conditions. Our results extend classical theory to the free probability framework.
\end{abstract}



\section{Introduction}

\subsection{Free stochastic differential equations.}\label{subsec:fSDE} 

A \textit{noncommutative probability space} is a pair $(\mathscr{A}, \tau),$ where $\mathscr{A}$ is a von Neumann algebra and $\tau: \mathscr{A} \to \mathbb{C}$ is a faithful, normal, tracial state (i.e. $\tau(\mathbf{1})=1$, $\tau$ is weakly continuous, and $\tau(XY)=\tau(YX)$ for all $X, Y\in\mathscr{A}$, with $\tau(X^*X)=0$ only if $X=0$). Introduced by Voiculescu \cite{D94}, free independence of operators in $\mathscr{A}$ offers a noncommutative counterpart to classical probability theory, leading to what is called free probability theory. Subsequently, many notions from classical probability were successfully transferred to free probability. We refer to the monographs \cite{NS2006, mingoFreeProbabilityRandom2017} for more details. 

\textit{Free stochastic differential equations} (SDEs), the free probability analogues of classical SDEs \cite{oksendal2003}, provide a natural framework for describing the evolution of operator-valued processes $\{U_t\}_{t\ge 0}.$  
In general, for a fixed $l\in\mathbb{N}^{*}$, we consider a class of free SDEs of the form
\begin{equation}\label{eq:fSDE}
dU_t=\alpha(U_t)\,{\rm d}t+\sum_{i=1}^l \beta^i(U_t)\,{\rm d}W_t\,\gamma^i(U_t),
\end{equation}
with an initial value $U_0\in \mathscr{A}_{\mathrm{sa}}$, where $\mathscr{A}_{\mathrm{sa}}$ denotes the set of self-adjoint operators in $\mathscr{A}.$  Note that the drift coefficient $\alpha$ and the diffusion coefficients $\beta^i, \gamma^i$ are operator-valued maps; $\{W_t\}_{t \ge 0}$ is a free Brownian motion, and the stochastic integral is defined by free stochastic calculus, which was introduced by Biane and Speicher   \cite{speicher1990, kummerer1992, biane1997, bianeStochasticCalculusRespect1998}. We refer to Section \ref{sec:pre} for more details. 

The motivation for studying free SDEs comes from random matrix theory. Let $f \in C^2(\mathbb{R})$ be a potential function. Consider the following $N \times N$ matrix-valued diffusion equation:
\begin{equation}
d U_t^{(N)} =-\frac{1}{2} f'(U_t^{(N)}) {\rm d}t + {\rm d}W_t^{(N)},
\end{equation}
where $\{W_t^{(N)}\}_{t\geq 0}$ is  a Hermitian matrix-valued Brownian motion. It can be proved that the large-$N$ limit of $U_t^{(N)}$ satisfies the following free diffusion equation  \cite{bianeFreeDiffusionsFree2001, guionnet_free_2009}: 
 \begin{equation}\label{eq:fDiff}
d U_t = -\frac{1}{2}f'(U_t)  {\rm d} t + {\rm d}W_t.
\end{equation}
Hence, it provides a hydrodynamic proof of the famous Wigner semi-circle law (for $f(x) = x^2/2$). On the other hand, for such matrix dynamics, the unique invariant measure, which is called the Gibbs state, is characterized by a classical entropy maximization principle. In the large-$N$ limit, the stationary distribution of the limiting free diffusion equations corresponds to the free Gibbs state, and the asymptotic analysis of the associated classical entropy leads directly to Voiculescu’s free entropy \cite{voiculescu_analogues_1993, voiculescu_analogues_1996, voiculescu_analogues_1997, voiculescu_analogues_1998, voiculescu_analogues_1999}. Thus, the study of free diffusions provides a dynamic pathway to the free Gibbs state and offers a natural stochastic interpretation of free entropy. Building  on this line of research, free diffusion equations \eqref{eq:fDiff} play an important role in many aspects of random matrix theory and free probability, such as large deviations for random matrices \cite{Biane:2003aa, Guionnet-prob-survey}, free entropy \cite{bianeFreeDiffusionsFree2001, biane2003, shlyakhtenko_lower_2009}, and free optimal transport theory \cite{Guionnet:2014aa, dabrowski_free_2021}. 
In particular, Dabrowski  \cite{dabrowski_non-commutative_2010, dabrowski_free_2013, dabrowski_time_2014, dabrowski_laplace_2017} developed a noncommutative path space framework for free diffusion equations, and investigated their stationary solutions, time-reversal symmetry, and deep connections with free entropy.

Thus, motivated by the free diffusion case, it is natural to consider free SDEs with general coefficients by addressing the following two questions: 1). \textit{Convergence}: How can we establish the convergence of matrix-valued SDEs to their free counterparts? To the best of our knowledge, only a few cases, apart from the diffusion case itself, have been rigorously confirmed \cite{Kemp2016, MP-AIHP-2002}. 2). \textit{Well-posedness}: What can be said about the well-posedness of free SDEs? Kargin \cite{karginFreeStochasticDifferential2011} proved the existence of local solutions to \eqref{eq:fSDE} provided the coefficients are locally operator Lipschitz. However, this condition does not, in general, guarantee the existence of a global solution. As mentioned earlier for the diffusion case, a global solution and its stationary distribution provide a stochastic interpretation of free entropy. Furthermore, establishing global solutions is crucial for numerical analysis, as their existence offers an a priori guarantee for the convergence of numerical methods \cite{schluechtermannNumericalSolutionFree2022a, wei_stochastic_2025}.

In this paper, we take a small step toward addressing the second question. Inspired by classical theory \cite{claudia_concise_2007}, we provide sufficient conditions for the existence of a global solution and a stationary distribution for the free SDE \eqref{eq:fSDE}.

\subsection{Global well-posedness of free SDEs.} 

Let $\alpha,\beta^i,\gamma^i$ ($i=1,\ldots,l$) be operator-valued maps such that $\alpha,\beta^i,\gamma^i: \mathscr{A}_{sa} \to \mathscr{A}_{sa}$. For any given initial value $U_0 \in \mathscr{A}_{sa}$ that is free from the free Brownian motion $\{W_t\}_{t\geq 0},$ a \textit{global solution} to \eqref{eq:fSDE} is defined as an operator-valued process $t \mapsto U_t$ from $\mathbb{R}_+$ into $\mathscr{A}_{sa}$ satisfying the integral equation  
\begin{equation}\label{eq:solution}
U_t = U_0 +\int_{0}^{t} \alpha(U_s ) {\rm d} s + \sum_{i=1}^l \int_{0}^{t} \beta^i (U_s ) {\rm d}W_s \gamma^i(U_s)
\end{equation}
for any $t \geq 0.$

\begin{remark}
For each $t\ge0,$ define
\begin{equation*}
\mathscr A_{\le t}:=W^*(U_0,W_r:r\le t),\quad
\mathscr A_{\ge t}:=W^*(U_t,W_r-W_t:r\ge t),\quad
\mathscr A_{=t}:=W^*(U_t)
\end{equation*}
which denote respectively the past, present and future of the process \cite{bianeFreeDiffusionsFree2001, gao2008}. It is natural to assume that the maps $\alpha, \beta^i, \gamma^i (i=1, \ldots, l)$ do not break causality of the process, meaning that operators in the past cannot escape their subalgebras via these maps. For instance, maps given by noncommutative polynomials in one-variable, or those derived using functional calculus, satisfy this property.
\end{remark}

\begin{remark}\label{rmk:sa}
In this paper, we restrict ourselves to self-adjoint solutions, motivated by the matrix case (since we only consider the Hermitian matrix model). Therefore, we require that both $\alpha(U_t)$ and the 
sum $\sum_{i=1}^l \int_{0}^{t} \beta^i (U_s ) {\rm d}W_s \gamma^i(U_s)$ be self-adjoint. 
\end{remark}

Recall the free diffusion equation
\begin{equation}\label{eq:free-diffusion}
d U_t = \alpha (U_t) {\rm d} t + {\rm d}W_t.
\end{equation}
Assume that $\alpha$ is locally operator Lipschitz \cite{Aleksandrov_2016}, 
namely, if for all $M > 0,$ 
there is a constant $L_\alpha(M) > 0$ such that
\begin{equation}\label{eq:local-operator-Lip}
\left\| \alpha (U) - \alpha(V) \right\| \leq L_\alpha(M) \left\| U- V \right\|
\end{equation}
for any $U, V \in \mathscr{A}_{\mathrm{sa}}$ and $\| U \|, \| V \| \leq M.$ Moreover, if the map $\alpha$ also satisfies the following condition:
\begin{equation}\label{eq:speicher-biane}
U\alpha(U) + \alpha(U)U + \mathbf{1} \le a U^2 + b \qquad (a\in\mathbb{R},\; b\ge0),
\end{equation}
for any $U \in \mathscr{A}_{sa}.$ Then Biane and Speicher proved there exists a unique global solution to \eqref{eq:free-diffusion} (see \cite[Theorem 3.1]{bianeFreeDiffusionsFree2001}). 
Their result was later extended to the equation
\begin{equation}\label{eq:fSDE-one-side}
d U_t =\alpha (U_t) {\rm d} t +\sum_{i=1}^l \left(\beta^i (U_t)  {\rm d}W_t +  {\rm d}W_t  \beta^i(U_t)\right)
\end{equation}
by Capitaine and Donati‑Martin \cite[Proposition A.1]{capitaine_free_2005}, again yielding a unique global solution under locally operator Lipschitz coefficients.
Regrading the general case \eqref{eq:fSDE}, Kargin \cite{karginFreeStochasticDifferential2011} proved that, under the local operator Lipschitz assumptions, one can only guarantee the existence of a local solution on a small time interval. 
\begin{theorem}\cite[Theorem 3.1]{karginFreeStochasticDifferential2011}\label{thm:solution-kargin}
Suppose that $\alpha, \beta^i,$ and $\gamma^i$ $(i=1,2,\ldots,l)$ are locally operator Lipschitz. Then there exist a sufficient small $T_0>0$ and a family of operators $U_t$ defined for
all $t \in [ 0, T_0),$ such that $U_t$ is a unique solution to \eqref{eq:fSDE} for $t < T_0.$
\end{theorem}
The key idea of their proofs is to show the convergence of the following Picard iteration: Define $U_t^{(0)}= U_0,$ and 
\begin{equation}\label{eq:picard-iter}
U_t^{(n+1)} = U_0 +\int_{0}^{t} \alpha(U_s^{(n)} ) {\rm d} s + \sum_{i=1}^l \int_{0}^{t} \beta^i (U_s^{(n)} ) {\rm d}W_s \gamma^i(U_s^{(n)}).
\end{equation}
However, in the noncommutative space $(\mathscr{A}, \tau)$, the Picard iteration may generate uncontrollable higher‑order terms arising from the diffusion coefficients, which prevents the extension of a local solution to a global one.

To overcome this obstacle, we allow the solution to reside in a slightly larger space, namely the noncommutative $L^2$-space $L^2(\mathscr{A}, \tau)$. This space is a Hilbert space with inner product $\langle U, V \rangle := \tau(U^* V)$ and is obtained as the completion of $\mathscr{A}$ with respect to the noncommutative $L^2$-norm $\|U\|_2 := (\tau(U^*U))^{1/2}$. We remark that the stochastic integration can also be extended to this framework (see \cite{anshelevichItoFormulaFree2002, gao2008,schluechtermannNumericalSolutionFree2022a}). Accordingly, let $L^2(\mathscr{A},\tau)_{sa}$ denote the set of all self‑adjoint operators in $L^2(\mathscr{A},\tau)$. We consider the solution $\{U_t\}_{t \ge 0}$ as an $L^2(\mathscr{A},\tau)_{sa}$-valued process satisfying the integral equation \eqref{eq:solution}. We begin to present our results by introducing the following assumptions on the drift and diffusion coefficients:

An operator-valued map $f: \mathscr{A}_{sa} \to \mathscr{A}_{sa}$ is called (locally) operator Lipschitz in the $L^2$-norm \cite{bianeFreeDiffusionsFree2001, gao2008}, if for all $t\ge0$ such that $f: \mathscr{A}_{sa} \to \mathscr{A}_{sa}$ and
for all $M > 0,$ there is a constant $L_f(M) > 0$ such that
\begin{equation}\label{eq:beta-gamma-Lip-loc}
\left\| f(U) - f(V) \right\|_2 \leq L_f(M) \left\| U- V \right\|_2
\end{equation}
for any $U, V \in  \mathscr{A}_{sa}$ with $\|U\|_2, \| V\|_2\leq M.$

\begingroup
\setcounter{tmp}{\value{assumption}}
\setcounter{assumption}{0} 
\renewcommand{\theassumption}{(A-\arabic{assumption})}

\begin{assumption}\label{assumption:ass1}
$\alpha,$ $\beta^i,$ and $\gamma^i$ ($i=1,\ldots,l$) are locally operator Lipschitz in the $L^2$-norm with the Lipschitz constants $L_{\alpha}(M),$ $L_{\beta^i}(M),$ and $L_{\gamma^i}(M)$, respectively. 
\end{assumption}
We have the following direct consequence of the above assumption: for any $U \in L^2(\mathscr{A},\tau)_{sa}$ with $\|U\|_2\le M$,
\begin{equation}\label{eq:bound-loc}
\|\alpha(U)\|_2 \le K_{\alpha}(M), \quad 
\|\beta^i(U)\|_2 \le K_{\beta^i}(M), \quad 
\|\gamma^i(U)\|_2 \le K_{\gamma^i}(M),
\end{equation}
where
\begin{equation*}
K_\alpha(M):=\|\alpha(0)\|_2+L_{\alpha}(M)M,
\end{equation*}
and similarly for $K_{\beta^i}(M)$ and $K_{\gamma^i}(M)$.

\begin{assumption}\label{assumption:ass2}	
There exist constants $ L_1\in\mathbb{R}$ and $L_2>0,$ such that for any $U\in  L^2(\mathscr{A},\tau)_{sa},$ 
	\begin{equation}\label{eq:ohtj}
	2\tau\left( U \alpha(U) \right)+(\sum_{i=1}^l \left\|\beta^i(U) \right\|_2 \cdot \left\|\gamma^i(U) \right\|_2)^2 \leq L_1\left\| U \right\|_2^2+L_2.
	\end{equation}
\end{assumption}
\endgroup

\begin{remark}
Given Assumption \ref{assumption:ass1}, by using the Picard iteration, one can straightforwardly establish the existence of a unique local solution via an adaptation of Kargin’s proof (see \cite{schluechtermannNumericalSolutionFree2022a}).
\end{remark}

Our first main result gives a sufficient condition for the existence and uniqueness of the global solution to the free SDE \eqref{eq:fSDE}.
\begin{theorem}\label{thm:solution}
Assume that Assumptions \ref{assumption:ass1} and \ref{assumption:ass2} hold. Then for any initial value $U_0 \in \mathscr{A}_{sa}$ that is free from $\{W_t\}_{t\geq 0},$ there exists a unique (global) solution $\{U_{t}\}_{t\in [0, \infty)}$ to the free SDE \eqref{eq:fSDE} such that $U_{t} \in L^2(\mathscr A,\tau)_{sa}$ for all $t\ge 0$.
Moreover, the solution is uniformly bounded in the $L^2$-norm and the map $t \to U_{t}$ is $L^2$-norm continuous.
\end{theorem}

\begin{remark}
There is another type of driving noise, namely, the free Lévy process \cite{anshelevichItoFormulaFree2002, gao2008}. Gao~\cite{gao2008} studied free diffusion equations driven by free L\'evy noise, establishing existence results within this framework.  We remark that results analogous to those in Theorem \ref{thm:solution} could be obtained for equations driven by free L\'evy noise.
\end{remark}

Our proof is based on the following \textit{Euler approximation} (we refer to \cite{claudia_concise_2007} for the classical case): For each $ n \in \mathbb{N},$ we construct a process $ U^{(n)}_t $ defined on the time interval $[0,\infty)$ by setting $ U^{(n)}_0 = U_0 $ and, for $ k \in \mathbb{N} \cup \{0\} $ and $ t \in (\frac{k}{n}, \frac{k+1}{n}] ,$
\begin{equation}
U_t^{(n)} = U_{\frac{k}{n}}^{(n)} + \int_{\frac{k}{n}}^{t} \alpha\bigl(U^{(n)}_{\frac{k}{n}}\bigr) \, {\rm d} s 
+ \sum_{i=1}^{l} \int_{\frac{k}{n}}^{t} \beta^i\bigl(U^{(n)}_{\frac{k}{n}}\bigr) \, {\rm d}W_s \, \gamma^i\bigl(U^{(n)}_{\frac{k}{n}}\bigr).
\end{equation}
Compared to the Picard iteration \eqref{eq:picard-iter}, the Euler scheme is more compatible with the free It\^o calculus. Specifically, the (noncommutative) martingale property of the free stochastic integral causes the diffusion term to vanish in the trace when estimating the $L^2$-norm, thereby eliminating the uncontrolled growth that would otherwise arise from diffusion coefficients. 

The proof proceeds in three steps. First, by Assumption \ref{assumption:ass1}, we localize the approximation by introducing a stopping time $T^{(n)}_M$ that keeps the process within a $L^2$-norm ball of radius $M>0.$ On the time interval $[0,T^{(n)}_M],$ an application of the free It\^o formula shows that $\{U^{(n)}_t\}_{n\geq 0}$ forms a Cauchy sequence in the noncommutative $L^2$-space. The limit of this sequence is then verified—using the free It\^o calculus—to satisfy the free SDE \eqref{eq:fSDE} locally in time, yielding a local solution. Finally, the Lyapunov‑type condition in Assumption \ref{assumption:ass2} guarantees that $M$ can be taken arbitrarily large. As a consequence, the stopping time tends to infinity, which extends the local solution to a global solution. Uniqueness of the global solution can subsequently be established. We refer to Section \ref{sec:sol} for more details. 

Furthermore, we can establish that the solution is a free Markov process.
Let $\mathscr{B} \subset \mathscr{A}$ be a von Neumann subalgebra, there exists a unique conditional expectation from $\mathscr{A}$ to $\mathscr{B}$ with respect $\tau$ \cite{pisier_non-commutative_1997}, denoted $\tau(\cdot \mid \mathscr{B}),$ such that for any $X \in \mathscr{A}$ and $Y, Z \in \mathscr{B},$
\begin{itemize}
\item $\tau (XY) = \tau (\tau (X\mid \mathscr{B}) Y);$
\item $\tau (YX Z \mid \mathscr{B}) = Y \tau (X \mid \mathscr{B}) Z.$
 \end{itemize}
Let $\mathscr{B}_1,\mathscr{B}_2$ are  von Neumann subalgebras of $\mathscr{A}$ and $\mathscr{B}\subseteq\mathscr{B}_1\cap\mathscr{B}_2.$ We say that the subalgebras $\mathscr{B}_1$ and $\mathscr{B}_2$ are \textit{free with amalgamation over $\mathscr{B}$}, or $\mathscr{B}_1$ and $\mathscr{B}_2$ are $\mathscr{B}$-free, if
\begin{equation}
\tau(X_1 X_2 \cdots X_n \mid \mathscr{B}) = 0 \; \text{for all} \; n \in \mathbb{N}, 
\end{equation}
whenever $X_j \in \mathscr{B}_{i_j}$, $i_j \in \{1, 2\}$, $i_1 \neq i_2 \neq \cdots \neq i_n$, and $\tau(X_j \mid \mathscr{B}) = 0$ (see Section 5.3 in \cite{voiculescu2000lectures}).

Let $\widetilde{\mathscr{A}}$ be the collection of all closed, densely defined operators affiliated with $\mathscr{A}$ \cite{pisier_non-commutative_1997}.  For any $U \in \widetilde{\mathscr{A}}$, let $U = V|U|$ be its polar decomposition, and define $W^*(U)$ as the von Neumann subalgebra of $\mathscr{A}$ generated by $V$ and $W^*(|U|)$.
We first recall the notion of free Markov process.

\begin{definition}[\cite{voiculescu2000lectures,voiculescu_analogues_1999,gao2008}]
Let $\{U_t\}_{t\ge0}$ be a family of self-adjoint operators in $\widetilde{\mathscr A}$.
For each $t\ge0$, define von Neumann subalgebras
\begin{equation*}
\mathscr B_{\le t}:=W^*(U_s:0\le s\le t),\qquad
\mathscr B_{\ge t}:=W^*(U_s:s\ge t),\qquad
\mathscr B_{=t}:=W^*(U_t).
\end{equation*}
We say that $\{U_t\}_{t\ge0}$ is a \emph{free Markov process} if, for every $t\ge0$, $\mathscr B_{\le t}$ and $\mathscr B_{\ge t}$ are $\mathscr B_{=t}$-free.
\end{definition}

Now suppose that $\{U_t\}_{t\ge0}$ is the unique solution to the free SDE \eqref{eq:fSDE}. Recall that for each $t\ge0,$
\begin{equation*}
\mathscr A_{\le t}=W^*(U_0,W_r:r\le t),\quad
\mathscr A_{\ge t}=W^*(U_t,W_r-W_t:r\ge t),\quad
\mathscr A_{=t}=W^*(U_t).
\end{equation*}
Then we can show that $\mathscr A_{=t} \subset \mathscr A_{\le t} \cap \mathscr A_{\ge t}.$ Therefore, by Lemma 2.1 in \cite{bianeFreeDiffusionsFree2001}, $\mathscr A_{\le t}$ and $\mathscr A_{\ge t}$ are $\mathscr A_{=t}$-free, which leads to the following proposition. We postpone its proof to Section \ref{subsec:free-markov-process}.

\begin{proposition}\label{prop:free-Markov}
Assume the hypotheses of Theorem~\ref{thm:solution}, and let $\{U_t\}_{t\geq 0}$, with $U_t \in L^2(\mathscr{A},\tau)\subset \widetilde{\mathscr{A}}$, be the unique solution to \eqref{eq:fSDE}. 
Then $\{U_t\}_{t\geq 0}$ is a free Markov process.
\end{proposition}

\begin{remark}
Let $C_b(\R)$ be all bounded continuous functions on $\R$. If the solution $U_t \in \mathscr{A_t}$, then there exists a family of transition operators (see \cite[Section 4]{bianeFreeDiffusionsFree2001})
\begin{equation*}
P_{s,t}:C_b(\R)\to C_b(\R),\qquad 0\le s\le t,
\end{equation*}
such that for every $f\in C_b(\R)$,
\begin{equation*}\label{eq:trans-oper}
\tau\!\left(f(U_t)\mid \mathscr B_{\le s}\right)
=
\tau\!\left(f(U_t)\mid \mathscr B_{=s}\right)
=
\bigl(P_{s,t}f\bigr)(U_s).
\end{equation*}
Moreover, since the free SDE are autonomous, we can also obtain that the transition operators
$\{P_{s,t}\}_{0\le s\le t}$ is time-homogeneous, i.e., 
$P_{s,t}=P_{t-s},$ $0\le s\le t.$
Consequently, if we set $P_t:=P_{0,t},$ $t\ge0,$
then $\{P_t\}_{t\ge0}$ is a Markov semigroup on $C_b(\mathbb{R})$.
\end{remark}


\subsection{Stationary distribution for free SDEs.} 

Our second main result concerns stationary distributions for free SDEs. In contrast, existing results \cite{guionnet_free_2009, shlyakhtenko_lower_2009, dabrowski_free_2021} are mainly restricted to the free diffusion equation \eqref{eq:fDiff} and its variants. In particular, for the free diffusion equation \eqref{eq:fDiff} with $f(x)=\frac{1}{2}x^2+\frac{g}{4}x^4,$ where $g$ is a negative constant sufficiently close to zero, Biane and Speicher \cite{bianeFreeDiffusionsFree2001} established the existence and uniqueness of a global solution, as well as its uniform boundedness in the operator norm, and conjectured the existence of a stationary distribution for such equations. Later, Guionnet and Shlyakhtenko's work provided an affirmative answer to this question (see \cite[Theorem 2.2]{guionnet_free_2009}). In this section, we provide sufficient conditions for the existence of a stationary distribution, thereby extending Guionnet and Shlyakhtenko's work to the general equation \eqref{eq:fSDE}. First, we recall some notions and definitions. 

Let $B_b(\mathbb{R})$ be the set of all bounded Borel functions on $\mathbb{R}$. 
For any $U \in \mathscr{A}_{sa}$, there exists a unique probability measure $\mu_U$ on $\mathbb{R}$ such that
for every $f \in B_b(\mathbb{R})$
\begin{equation*}
\tau\bigl(f(U)\bigr) = \int_{\mathbb{R}} f(x) \, \mu_U({\rm d}x).
\end{equation*}
The measure $\mu_U$ is called the \textit{distribution} of $U$. 
Let $\mathcal P(\mathbb R)$ denote the space of all Borel probability measures on $\mathbb R$ and $\mathcal P_2(\mathbb R)$ the space of probability measures on $\mathbb R$ with finite second moment, i.e.,
\begin{equation*}
\mathcal P_2(\mathbb R)=\left\{\mu \in \mathcal P(\mathbb R):\int_{\mathbb R} |x|^2 \,\mu({\rm d}x) < \infty\right\}.
\end{equation*}
For $\mu,\nu \in \mathcal P_2(\mathbb R)$, the free Wasserstein metric \cite{Biane2001Wasserstein} between $\mu$ and $\nu$ is defined by
\begin{equation*}
W_2(\mu,\nu)=\inf\left\{\|X-Y\|_2:\mu_X = \mu,\ \mu_Y = \nu\right\},
\end{equation*}
where the infimum is taken over all self‑adjoint operators $X, Y$ in a tracial $W^*$-probability space with the prescribed distributions. We remark that the free Wasserstein metric coincides with the classical Wasserstein metric for the one-variable case.
We first give the definition of the stationary distribution for free SDEs.

\begin{definition}[Stationary distribution]
Suppose that $\{U_t\}_{t\ge 0}$ is a solution to the free SDE \eqref{eq:fSDE}. A probability measure $\mu\in\mathcal P(\mathbb R)$ is called a \textit{stationary distribution} for the solution (or the equation) if, for every $t\ge0$, 
\begin{equation}\label{eq:stationary}
\tau(f(U_t))=\int_{\mathbb R} f(x)\,\mu({\rm d} x),
\qquad \forall f\in B_b(\mathbb R).
\end{equation}
\end{definition}

Our results rely on two main assumptions. The first one is that we require the {\it weak uniqueness} of the equation. Note that the uniqueness obtained in Theorem \ref{thm:solution} referred to as \textit{strong uniqueness}: whenever $\{U_t\}_{t\ge0}$ and $\{\widehat{U}_t\}_{t\ge0}$ are two operator-valued processes satisfying \eqref{eq:solution} with the same driving free Brownian motion and the same initial value, one has $U_t=\widehat{U}_t$ for all $t\ge0.$ If, on the other hand, we are only given the drift and diffusion coefficients, and seek a tuple $\{U_t, W_t\}$ in some noncommutative probability space $(\mathscr{A}, \tau)$ such that \eqref{eq:solution} holds, then $\{U_t, W_t\}_{t\ge0}$ is called a \textit{weak} solution to \eqref{eq:fSDE}. 

\begin{definition}
We say that the free SDE \eqref{eq:fSDE} admits \textit{weak uniqueness} if the following holds: for any two weak solutions $\{U_t, W_t\}_{t\geq 0}$ and $\{\widehat{U}_t, \widehat{W}_t\}_{t \geq 0}$ to \eqref{eq:fSDE}, whenever $U_0$ and $\widehat{U}_0$ have the same distribution, then $U_t$ and $\widehat{U}_t$ have the same distribution for all $t \geq 0.$
\end{definition}

\begin{remark}
Assume that Assumptions \ref{assumption:ass1} and \ref{assumption:ass2} hold. One can prove that the free SDE \eqref{eq:fSDE} admits \textit{weak uniqueness}. We refer to Proposition \ref{prop:same-law} and its proof. 
\end{remark}

The second assumption concerns the required conditions on the drift and diffusion coefficients.

\begingroup
\edef\originalAssumptionFormat{\theassumption}
\newcounter{assumptionbackup}
\setcounter{assumptionbackup}{\value{assumption}}

\setcounter{assumption}{0}
\renewcommand{\theassumption}{(B-\arabic{assumption})}

\begin{assumption}\hypertarget{assumption:ass3}{}\label{assumption:ass3}
For all $M >0,$ 
there exists a constant $L'_\alpha \in \mathbb{R}$ (depends on $\alpha$ and $M$), such that
\begin{equation}\label{eq:alpha-onesideLip-loc}
\langle U-V,\alpha(U)-\alpha(V)\rangle\leq L_\alpha' \|U -V\|^2_2
\end{equation}
for any operators $U, V \in \mathscr{A}_{sa}$ with $\|U\|, \|V\| \leq M$.

Moreover, assume that the diffusion coefficients $\beta^i$ and $\gamma^i$ ($i=1,\ldots,l$) are locally operator Lipschitz in the $L^2$-norm with the Lipschitz constants $L_{\beta^i}$ and $L_{\gamma^i}$, respectively. 
\end{assumption}

\setcounter{assumption}{\value{assumptionbackup}}
\renewcommand{\theassumption}{\originalAssumptionFormat}
\endgroup

Now we are ready to present our result on the stationary distribution.

\begin{theorem}\label{thm:stationary}
Suppose that the free SDE \eqref{eq:fSDE} admits weak uniqueness, and for any free Brownian motion $\{W_t\}_{t \geq 0}$ and initial value $U_0 \in \mathscr{A}_{sa}$ that is free from $\{W_t\}_{t\geq 0},$ there exists a unique strong solution $\{U_t\}_{t\ge 0}$ to \eqref{eq:fSDE} such that $U_t \in \mathscr{A}_{\le t}$ and $\sup_{t\geq 0} \|U_t\|< \infty$. Assume further that Assumption \hyperlink{assumption:ass3}{(B-1)} holds. If 
\begin{equation}\label{eq:stationary-condition}
L'_\alpha < -\frac{1}{2}(\sum_{i=1}^{l} (L_{\beta^i} K_{\gamma^i}+ L_{\gamma^i} K_{\beta^i}))^2,
\end{equation}
then the solution $\{U_t\}_{t\ge 0}$ admits a unique stationary distribution $\mu$. Moreover, for any initial value, the distribution of $U_t$ converges to $\mu$ in $W_2$ as $t \to \infty$.
\end{theorem}

\begin{remark}\label{rmk:fdiff-weak}
The free diffusion equation \eqref{eq:fDiff} with $f(x)=\frac{1}{2}x^2+\frac{g}{4}x^4$ admits weak uniqueness (see Remark \ref{rmk:weak-unique-free-diffusion} in Section \ref{subsec:weak-unique}). 
Moreover, it is easy to check that Assumption \ref{assumption:ass3} holds for this case (see Section \ref{sec:exa}). Therefore, we provide an alternative answer to Biane and Speicher’s question for the one-variable case. 
\end{remark}

For the proof of Theorem \ref{thm:stationary}, we begin by proving that the perturbation of solution is bounded by the perturbation of initial value, a property we call {\it stability}. The stability of solutions then implies that the distribution of the solution $\{U_t\}_{t\ge0}$ is Cauchy in $W_2$, and hence converges to a limiting distribution $\mu.$ Finally, Theorem \ref{thm:stationary} is established by showing that for any initial value has the distribution $\mu$, the distribution of the corresponding solution $\{U_t\}_{t\ge0}$ remains $\mu$; what is, $\mu$ is a stationary distribution. 

As a corollary, we can prove that the distribution of the solution $U_t$ converges to the stationary distribution at an exponential rate.
\begin{corollary}\label{cor:exp-conv}
Assume the hypotheses of Theorem~\ref{thm:stationary}. Let $\mu$ be the stationary distribution for the solution $\{U_t\}_{t\geq 0}.$ Then for every $ f \in C_c^3(\mathbb{R}) $, there exist constants $ C_f, C >0$, independent of $ t $, such that 
\begin{equation}\label{eq:exp-bound}
\bigl| \tau\bigl(f(U_t)\bigr) - \int_{\mathbb{R}} f(x) \, \mu( {\rm d} x) \bigr| \le C_f e^{-C t} \|U_0\|_2.
\end{equation}
\end{corollary}


\subsection{Organization of the paper.} The remainder of the paper is organized as follows. Section \ref{sec:pre} collects the necessary preliminaries on free stochastic calculus and free It\^o formula. Section \ref{sec:sol} presents the proof of the existence and uniqueness of a global solution (Theorem \ref{thm:solution}) and shows the solution is a free Markov process (Proposition~\ref{prop:free-Markov}). Section \ref{sec:stat} is devoted to the study of long‑time behavior and contains the proof of Theorem \ref{thm:stationary} on the existence and uniqueness of a stationary distribution and the exponential convergence estimate (Corollary \ref{cor:exp-conv}). Finally, Section \ref{sec:exa} provides concrete examples that satisfy the assumptions of our theorems.

\section{Preliminaries}\label{sec:pre}
In this section, we recall some basic notions and tools for studying free SDEs. Let $(\mathscr{A},\tau)$ be a noncommutative probability space. 

\begin{definition}[Free independence \cite{NS2006, mingoFreeProbabilityRandom2017}]
Subalgebras $\mathscr{A}_1,\dots,\mathscr{A}_m \subset \mathscr{A}$ are said to be \emph{freely independent} if for every $n\in\mathbb{N}$, every choice of indices $i_1,\dots,i_n \in \{1,\dots,m\}$ with $i_k \neq i_{k+1} \text{ for all } k,$
and every choice of operators $X_k \in \mathscr{A}_{i_k}$ satisfying $\tau(X_k)=0$ for all $k=1,\dots,n$, one has
\begin{equation*}
\tau(X_1\cdots X_n)=0.
\end{equation*}
\end{definition}

\begin{definition}[Semicircular element]
A self‑adjoint operator $W \in \mathscr{A}$ is called a \emph{semicircular element} with mean $0$ and variance $\sigma^2 > 0$ if its moments are given by the semicircular law $\mu_{\mathrm{sc}}$, i.e.,
\begin{equation*}
\tau(W^k) = \int_{\mathbb{R}} t^k \, \mu_{\mathrm{sc}}( {\rm d} t), \qquad 
\mu_{\mathrm{sc}}({\rm d}t) = \frac{\sqrt{4\sigma^{2}-t^{2}}}{2\pi\sigma^{2}} \, \mathbf{1}_{\{|t|\le 2\sigma\}} \, {\rm d}t.
\end{equation*}
\end{definition}

It is well known that semicircular elements are the free analogues of Gaussian random variables. Their continuous‑time counterpart, free Brownian motion, constitutes the fundamental driving noise in free stochastic calculus. 

\begin{definition}[Free Brownian motion]
A stochastic process $\{W_t\}_{t\ge 0}$ is called a \emph{free Brownian motion} if it satisfies 
\begin{itemize}
\item $W_0=0;$
\item The increment $W_t-W_s$ is freely independent from $W^*(W_r: 0\leq r\leq s)$ for all $0\leq s <t;$
\item The increment $W_t-W_s$ is a semicircular element with mean $0$ and variance $t-s$ for all $0 \leq s <t.$
\end{itemize}
\end{definition}


Recall that $\mathscr A_{\le t}=W^*(U_0,W_r:r\le t)$ where $U_0\in\mathscr A_{sa}$ is free from $\{W_t\}_{t\ge 0}.$ It is clear that $\{\mathscr A_{\le t}\}_{t \geq 0}$ is a filtration on $\mathscr{A}$, i.e., an increasing family of subalgebras
of $\mathscr{A}.$ 
Given $T\in[0,\infty),$ suppose that $ \beta_t,  \gamma_t \in \mathscr A_{\le t}$ and $\|\beta_t\| \cdot \|\gamma_t\| \in L^2([0,T])$. Let $\mathcal{P}: 0= t_0 \leq t_1 \leq \cdots \leq t_n=T$ be a partition of $ [0,T] $, denoted by $ \Delta_n $. 
Consider the finite sum 
\begin{equation*}
I_n:=\sum_{i=1}^{n}\beta_{t_{i-1}}(W_{t_i}-W_{t_{i-1}})\gamma_{t_{i-1}}.
\end{equation*}
It is known \cite{bianeStochasticCalculusRespect1998} that $I_n$ converges in the operator norm as $d(\Delta_n)\rightarrow 0,$ where $ d(\Delta_n):=\max_{1\le k\le n}(t_{k}-t_{k-1}),$  and the convergence does not depend on the partition. Hence, the \emph{free stochastic integral} \cite{bianeStochasticCalculusRespect1998,karginFreeStochasticDifferential2011} of $\{\beta_t\}_{t\geq 0}$ and $\{\gamma_t\}_{t\geq 0}$ is defined as the limit (in the operator norm) of $I_n$, denoted by
\begin{equation*}
I=\int_{0}^{T}\beta_t  {\rm d}W_t \gamma_t.
\end{equation*}

A key ingredient to show the convergence of $I_n$ is the following free Burkholder-Gundy (B-G) inequality \cite{bianeStochasticCalculusRespect1998}: 
\begin{equation}\label{eq:B-G}
\left\|\int_0^T \beta_t {\rm d}W_t \gamma_t \right\| \leq 2 \sqrt{2} \left( \int_{0}^T \|\beta_t\|^2 \cdot \|\gamma_t\|^2 {\rm d} t \right)^{1/2},
\end{equation}  
and for the $L^2$-norm, we have the following free It\^o isometry:
\begin{equation}\label{eq:isometry}
\left\|\int_0^T \beta_t {\rm d}W_t \gamma_t \right\|_2 =\left( \int_{0}^T \|\beta_t\|^2_2 \cdot \|\gamma_t\|^2_2 {\rm d} t \right)^{1/2}. 
\end{equation}  
 
The \textit{free It\^o formula} \cite{bianeStochasticCalculusRespect1998, anshelevichItoFormulaFree2002} plays the same role in free stochastic calculus as the classical It\^o formula does in classical stochastic calculus. Formally, the multiplication rules are
\begin{itemize}
\item $\alpha_t {\rm d} t \cdot \beta_t {\rm d} t  = \alpha_t {\rm d} t\cdot \beta_t  {\rm d}W_t \gamma_t = \alpha_t {\rm d}W_t \beta_t \cdot  \gamma_t {\rm d} t=0,$
\item $\alpha_t {\rm d}W_t \beta_t \cdot \gamma_t {\rm d}W_t \zeta_t  =\tau \left(\beta_t \gamma_t \right)  \alpha_t \zeta_t {\rm d}t$
\end{itemize}
for any $ \alpha_t, \beta_t, \gamma_t, \zeta_t \in \mathscr A_{\le t}$ such that free stochastic integrals are well-defined. Specifically, let 
\begin{equation*}
U_t =U_0 + \int_0^t \alpha_s {\rm d}t + \int_0^t \beta_t {\rm d}W_t \gamma_t.
\end{equation*}
The following free It\^o formula (in integration form) is used frequently in this paper:
\begin{equation}\label{eq:free-ito-polynomial}
\begin{split}
U_t^2 =  U_0^2 &+ \int_0^t  (U_s \alpha_s+ \alpha_s U_s) {\rm d}s + \int_0^t \tau (\beta_s \gamma_s) \beta_s \gamma_s {\rm d}s\\
& + \int_0^t U_s \beta_s {\rm d}W_s \gamma_s + \beta_s {\rm d}W_s \gamma_s U_s.
\end{split}
\end{equation}

\begin{remark}
Suppose that $ \beta_t,  \gamma_t \in L^2(\mathscr A_{\le t},\tau)$ and $\|\beta_t\|_2 \cdot \|\gamma_t\|_2 \in L^2([0,T])$. Then by the It\^o-isometry \eqref{eq:isometry}, the free stochastic It\^o integral $I = \int_0^T\beta_t\,{\rm d}W_t\,\gamma_t$ is well-defined in $L^2(\mathscr{A},\tau)$ (see \cite{bianeStochasticCalculusRespect1998, schluechtermannNumericalSolutionFree2022a}). 
\end{remark}

 
\section{Well-posedness of free SDEs}\label{sec:sol}

\subsection{The Euler scheme}

In the work of Biane and Speicher \cite{bianeFreeDiffusionsFree2001}, an Euler scheme was introduced for establishing certain regularity results for the diffusion equations. Motivated by the classical case \cite{claudia_concise_2007}, this scheme can also be employed to demonstrate the well-posedness of free stochastic differential equations. Without loss of generality, we set $l = 1$ in \eqref{eq:fSDE}; the general case follows readily by applying the triangle inequality. Recall the Euler scheme: Fix $n \in \mathbb{N},$ first in the interval $\bigl [0, \frac{1}{n} \bigr],$ we set $U^{(n)}_0 = U_0\in \mathscr{A}_{sa},$ and for $t \in \bigl(0, \frac{1}{n} \bigr]$, define
\begin{equation*}
 U_t^{(n)}=U_{0}+\int_{0}^{t}\alpha(U_0) {\rm d}s + \int_{0}^{t}\beta(U_0) {\rm d}W_s \gamma(U_0),
\end{equation*}
and then continue by induction on $k \in \mathbb{N},$ assuming $U_t^{(n)}$ has been defined for $t \leq \frac{k}{n},$ for $t \in \bigl(\frac{k}{n}, \frac{k+1}{n} \bigr],$ define 
\begin{equation}\label{eq:Euler-approx}
 U_t^{(n)}=U_{\frac{k}{n}}^{(n)}+\int_{\frac{k}{n}}^{t}\alpha(U^{(n)}_{\frac{k}{n}}) {\rm d}s + \int_{\frac{k}{n}}^{t}\beta(U^{(n)}_{\frac{k}{n}}) {\rm d}W_s \gamma(U^{(n)}_{\frac{k}{n}}).
\end{equation}

Setting $\kappa(n,t):=[tn]/n$, we write that for any $t\ge 0$
\begin{equation}\label{eq:Euler-approx-1}
\begin{split}
 U_t^{(n)}&=U_0+\int_{0}^{t}\alpha(U^{(n)}_{\kappa(n,s)}) {\rm d}s + \int_{0}^{t}\beta(U^{(n)}_{\kappa(n,s)}) {\rm d}W_s \gamma(U^{(n)}_{\kappa(n,s)})\\
 &=U_0+\int_{0}^{t}\alpha(U^{(n)}_s+p^{(n)}_s) {\rm d}s + \int_{0}^{t}\beta(U^{(n)}_s+p^{(n)}_s) {\rm d}W_s \gamma(U^{(n)}_s+p^{(n)}_s),
 \end{split}
\end{equation}
where 
\begin{equation}
\begin{split}
p^{(n)}_t&:=U^{(n)}_{\kappa(n,t)}-U^{(n)}_t\\
&=-\int_{\kappa(n,t)}^{t}\alpha(U^{(n)}_{\kappa(n,s)}) {\rm d}s - \int_{\kappa(n,t)}^{t}\beta(U^{(n)}_{\kappa(n,s)}) {\rm d}W_s \gamma(U^{(n)}_{\kappa(n,s)}).
\end{split}
\end{equation}

\begin{remark}\label{UinA}
Note that $ U_s^{(n)}\in (\mathscr A_{\le t})_{sa},$  for all $0\le s\le t$ and $t\ge 0.$
Indeed, this follows by induction over the mesh intervals.
For $0<s\le \frac1n\wedge t$, we have
\begin{equation*}
U_s^{(n)}
=
U_0+\int_0^s \alpha(U_0)\,{\rm d} r+\int_0^s \beta(U_0)\,{\rm d}W_r\,\gamma(U_0).
\end{equation*}
Since $U_0\in \mathscr A_{\le t}$ and $\alpha,\beta,\gamma$ do not break causality, we have $\alpha(U_0),\beta(U_0),\gamma(U_0)\in \mathscr A_{\le t}.$
Moreover, $W_r\in \mathscr A_{\le t}$ for all $r\le t.$ Consequently,
\begin{equation*}
U_s^{(n)}\in \mathscr A_{\le t}.
\end{equation*}

Now assume that $U_r^{(n)}\in \mathscr A_{\le t}, r\le \frac{k}{n}\le t.$
For $s\in \Bigl(\frac{k}{n},\frac{k+1}{n}\Bigr]\cap[0,t],$
the Euler approximation gives
\begin{equation*}
\begin{split}
U_s^{(n)}
&=
U_{\frac{k}{n}}^{(n)}
+\int_{\frac{k}{n}}^s \alpha\Bigl(U_{\frac{k}{n}}^{(n)}\Bigr)\,{\rm d}s
+\int_{\frac{k}{n}}^s
\beta\Bigl(U_{\frac{k}{n}}^{(n)}\Bigr)\,{\rm d}W_r\,\gamma\Bigl(U_{\frac{k}{n}}^{(n)}\Bigr)\\
&=
U_{\frac{k}{n}}^{(n)}
+\Bigl(s-\frac{k}{n}\Bigr)\alpha\Bigl(U_{\frac{k}{n}}^{(n)}\Bigr)
+\beta\Bigl(U_{\frac{k}{n}}^{(n)}\Bigr)\bigl(W_s-W_{\frac{k}{n}}\bigr)\gamma\Bigl(U_{\frac{k}{n}}^{(n)}\Bigr).
\end{split}
\end{equation*}
Since $W_s-W_{\frac{k}{n}}\in \mathscr A_{\le t}, s\le t,$ the induction hypothesis implies that $U_s^{(n)} \in \mathscr A_{\le t}$.
So by the induction, we have 
\begin{equation*}
U_s^{(n)}\in \mathscr A_{\le t},\qquad 0\le s\le t.
\end{equation*}
Moreover, the self-adjointness of $U_s^{(n)}$ follows from Remark~\ref{rmk:sa}.
\end{remark}
The proof of Theorem \ref{thm:solution}, which we present in the remainder of this section, is divided into the following subsections.


\subsection{Cauchy property of the Euler approximation} In this subsection, we will prove the following proposition.

\begin{proposition}\label{prop:Cauchy-prop}
Assume that Assumptions  \ref{assumption:ass1} and \ref{assumption:ass2} hold. Given $T\in[0,\infty)$, and let $U_t^{(n)}$ be defined in \eqref{eq:Euler-approx-1}. Then we have 
\begin{equation}\label{eq:cauchy}
\sup_{t\in[0,T]}\|U^{(n)}_t-U^{(m)}_t\|_2\to 0
\end{equation}
as $n,m\to \infty.$
\end{proposition}
The above proposition will be proved by using a stopping time argument. Fix $M\in[0,\infty)$ such that $\|U_0\|_2 \leq M/3,$ and define the stopping time
\begin{equation}\label{eq:stopping-time}
T_M^{(n)}=\inf\{t\ge0 : \|U_t^{(n)}\|_2 > M/3\}.
\end{equation}
Note that the map $t \mapsto \|U_t^{(n)}\|_2$ is continuous, we then have $T_M^{(n)} > 0$. Consequently,
\begin{equation}\label{eq:upper-bounds-U-p}
\|U_t^{(n)}\|_2\leq M/3 \;\; \text{and} \;\; \|p_t^{(n)}\|_2\leq 2M/3 
\end{equation}
for all $ t\in[0,T_M^{(n)}].$ Moreover, applying the free It\^o isometry \eqref{eq:isometry} we obtain
\begin{equation*}\label{eq:p-estimate}
\begin{split}
\|p^{(n)}_t\|_2 
&= \Bigl\|\int_{\kappa(n,t)}^{t}\alpha\bigl(U^{(n)}_{\kappa(n,s)}\bigr) {\rm d}s 
+ \int_{\kappa(n,t)}^{t}\beta\bigl(U^{(n)}_{\kappa(n,s)}\bigr) {\rm d}W_s \,\gamma\bigl(U^{(n)}_{\kappa(n,s)}\bigr)\Bigr\|_2\\
&\le \int_{\kappa(n,t)}^{t}\bigl\|\alpha\bigl(U^{(n)}_{\kappa(n,s)}\bigr)\bigr\|_2 {\rm d}s 
+ \Bigl(\int_{\kappa(n,t)}^{t}\|\beta\bigl(U^{(n)}_{\kappa(n,s)}\bigr)\|_2^2 
\, \|\gamma\bigl(U^{(n)}_{\kappa(n,s)}\bigr)\|_2^2 {\rm d}s \Bigr)^{\frac12}.
\end{split}
\end{equation*}
Note that $\kappa(n,t)\to t$ as $n\to\infty.$ It follows from \eqref{eq:bound-loc} that
\begin{equation}\label{eq:p-bounds}
\sup_{t\in[0,T_M^{(n)}]}\|p^{(n)}_t\|_2 \to 0
\end{equation}
as $n \to \infty.$

\begin{lemma}\cite{bianeStochasticCalculusRespect1998}\label{lem:tau0}
Let $X_t,Y_t$ be $\mathscr A_{\le t}$-adapted processes. If for any $T>0,$
\begin{equation*}
\int_0^T \|X_t\|_2^2 \cdot \|Y_t\|_2^2  \; {\rm d}t< \infty,
\end{equation*}
 then 
\begin{equation}\label{eq:martingale-prop}
\tau\!\left( \int_0^T X_t\, {\rm d}W_t \, Y_t \right) = 0.
\end{equation}
\end{lemma}

\begin{lemma}\label{lem:marti}
For $T\in[0,\infty)$, and $n\in\mathbb{N},$ let $\mathcal{P}_n: 0=t_0 < t_1<\cdots < t_{n-1} < t_n =T$ be 
an uniform partition of $[0,T]$ with mesh $1/n$ and $X^{(n)}_t$ be a piecewise-constant $\mathscr{A}_{\le t}$-adapted process such that
\begin{equation*}
X_t^{(n)} = X_{t_{j-1}}^{(n)}, \quad \forall t \in \bigl[t_{j-1},\, t_j\bigr), \; j = 1, \dots, n.
\end{equation*}

Similarly, for $m\in\mathbb{N},$ let $\mathcal{P}_m: 0=s_0 < s_1<\cdots < s_{m-1} < s_m =T$ be 
another uniform partition of $[0,T]$ with mesh $1/m.$ Let  $Y^{(m)}_t$ be a piecewise-constant $\mathscr{A}_{\le t}$-adapted process such that
\begin{equation*}
Y_t^{(m)} = Y_{s_{j-1}}^{(m)}, \quad \forall t \in \bigl[s_{j-1},\, s_j\bigr), \; j = 1, \dots, m.
\end{equation*}

If $Z_t$ is an $\mathscr A_{\le t}$-adapted process such that
\begin{equation*}
\sup_{t\in[0,T]}\|Z_t\|_2<\infty,
\end{equation*}
then for any $n,m\in\mathbb{N},$ we have 
\begin{equation}\label{eq:martingale-prop}
\tau\!\left( \int_0^T Z_s X_s^{(n)}\, {\rm d}W_s \, Y_s^{(m)} \right) =\tau\!\left( \int_0^T X_s^{(n)}\, {\rm d}W_s \, Y_s^{(m)} Z_s\right) = 0.
\end{equation}
\end{lemma}

\begin{proof}
We shall only prove that the first integral in \eqref{eq:martingale-prop} has zero mean, since the second integral can be handled in the same way.
Without loss of generality, we may assume that $n \leq m$ and $t_j \neq s_j, j =1, \ldots, n.$ Consider the common refinement $\mathcal{P}:0 = p_1 < p_2 \cdots < p_{k-1} < p_k =T$ of $\mathcal{P}_n$ and $\mathcal{P}_m$, where $k =n+m-2.$  
Therefore, for each $r=1,\dots,k$, the subinterval $[p_{r-1}, p_r)$ is contained in an unique subinterval of $\mathcal{P}_n$, say $\bigl[t_{i_r-1},\, t_{i_r}\bigr)$, and in an unique subinterval of $\mathcal{P}_m$, say $\bigl[s_{j_r-1},\, s_{j_r} \bigr)$, where $i_r \in \{1, 2,\ldots, n\}$ and $j_r \in \{1, 2,\ldots, m\}.$
Set
\begin{equation*}
A_{r-1}= X_{t_{i_r-1}}^{(n)}, \qquad 
B_{r-1}= Y_{s_{j_r-1}}^{(m)}.
\end{equation*}
Then for any $s \in [p_{r-1},p_r)$, we have 
\begin{equation*}
X^{(n)}_s \equiv A_{r-1}, \quad Y^{(m)}_s \equiv B_{r-1},
\end{equation*}
and both $A_{r-1}$ and $B_{r-1}$ belong to the subalgebra $\mathscr{A}_{\le p_{r-1}}$.

It follows that 
\begin{equation*}
\int_0^T Z_s X_s^{(n)}\, {\rm d}W_s \, Y_s^{(m)} 
= \sum_{r=1}^k  \int_{p_{r-1}}^{p_r} Z_s A_{r-1}\, {\rm d}W_s \, B_{r-1}.
\end{equation*}

Since for all $r=1,\ldots,k$
\begin{equation*}
\begin{split}
\int_{p_{r-1}}^{p_r} \|Z_s A_{r-1}\|_2^2 \|B_{r-1}\|_2^2 {\rm d} s&\leq \int_{p_{r-1}}^{p_r} \|Z_s\|_2^2 \|A_{r-1}\|^2 \|B_{r-1}\|_2^2 {\rm d}s\\
&\leq \sup_{t\in[0,T]}\|Z_t\|_2^2\|A_{r-1}\|^2 \|B_{r-1}\|^2(p_r-p_{r-1})\\
&<\infty,
\end{split}
\end{equation*}
 it follows from Lemma~\ref{lem:marti} that
 \begin{equation*}
\tau\!\left( \int_{p_{r-1}}^{p_r} Z_s A_{r-1}\, {\rm d}W_s \, B_{r-1}\right) = 0.
\end{equation*}
Summing over $r$ gives \eqref{eq:martingale-prop}. 
\end{proof}

Now we are ready to prove Proposition \ref{prop:Cauchy-prop}. Firstly, we introduce the following notations. 
\begin{notation}
For any $n,m\in\mathbb{N},$ we denote
\begin{itemize}
\item $\Delta_{n,m} [U_s]:=U^{(n)}_s -U^{(m)}_s;$  
\item $\Delta_{n,m} [p_s]:=p^{(n)}_s-p^{(m)}_s;$
\item $\Delta_{n,m}[\alpha_s]:=\alpha(U^{(n)}_s+p^{(n)}_s)-\alpha(U^{(m)}_s+p^{(m)}_s);$
\item $\Delta_{n,m}[\beta_s]:=\beta(U^{(n)}_s+p^{(n)}_s)-\beta(U^{(m)}_s+p^{(m)}_s);$
\item $\Delta_{n,m}[\gamma_s]:=\gamma(U^{(n)}_s+p^{(n)}_s)-\gamma(U^{(m)}_s+p^{(m)}_s).$
\end{itemize}
\end{notation}
Assume that Assumption \ref{assumption:ass1} holds. Note that $\|U^{(n)}_s+p^{(n)}_s\|_2,\|U^{(m)}_s+p^{(m)}_s\|_2\le M$ (see \eqref{eq:upper-bounds-U-p}). Thus, we have \begin{equation}\label{eq:loclip-alpha}
\|\Delta_{n,m}[\alpha_s]\|_2 \leq L_\alpha (M) \|\Delta_{n,m} [U_s] + \Delta_{n,m} [p_s]\|_2.
\end{equation}
Moreover, it follows from \eqref{eq:bound-loc} that  
\begin{equation}\label{eq:bound-Dalpha}
\|\Delta_{n,m}[\alpha_s]\|_2 \leq  \| \alpha(U^{(n)}_s+p^{(n)}_s)\|_2 + \|\alpha(U^{(m)}_s+p^{(m)}_s)\|_2 \leq 2 K_\alpha(M).
\end{equation}
We remark that similar bounds also hold for $\|\Delta_{n,m}[\beta_s]\|_2$ and $\|\Delta_{n,m}[\gamma_s]\|_2.$

\begin{lemma}\label{lem:Cauchy-prop-stopping-time}
Assume that Assumption \ref{assumption:ass1} holds. Let $T_M^{(n)}$ be the stopping time given in \eqref{eq:stopping-time}. Then we have 
\begin{equation}\label{eq:cauchy-stopping-time}
\sup_{t\in[0,T_M^{(n)}\wedge T_M^{(m)}]} \|\Delta_{n,m} [U_t]\|_2\to 0
\end{equation}
as $n,m\to \infty.$
\end{lemma}

\begin{proof}

For any $n,m\in\mathbb{N},$ it follows from \eqref{eq:Euler-approx-1} that
\begin{equation*}\label{eq:DU}
\begin{split}
\Delta_{n,m} [U_t]
&= \int_{0}^{t}\Delta_{n,m}[\alpha_s] {\rm d}s \\
& \quad \quad+ \int_{0}^{t}\Delta_{n,m}[\beta_s] {\rm d}W_s \,\gamma(U^{(n)}_{\kappa(n,s)})+ \int_{0}^{t} \beta(U^{(m)}_{\kappa(m,s)}) \,{\rm d}W_s \,\Delta_{n,m}[\gamma_s].
\end{split}
\end{equation*}

Let $C_1\ge0$ be a constant that will be determined later. For $t \in [0, T_M^{(n)} \wedge T_M^{(m)}]$, by the free It\^o formula~\eqref{eq:free-ito-polynomial}, we have
\begin{equation}\label{eq:DU-bound-ito}
\begin{split}
e^{-C_1t}&(\Delta_{n,m} [U_t])^2 \\
& = -C_1\int_{0}^{t}e^{-C_1s} (\Delta_{n,m} [U_s])^2 {\rm d}s \\
&\quad+ \int_{0}^{t}e^{-C_1s}\left( \Delta_{n,m} [U_s]\,\Delta_{n,m}[\alpha_s]+\Delta_{n,m}[\alpha_s] \,\Delta_{n,m} [U_s] \right) {\rm d}s \\
&\quad + \int_{0}^{t}e^{-C_1s}\left( \tau\!\left(\gamma(U^{(n)}_{\kappa(n,s)}) \,\Delta_{n,m}[\beta_s]\right)\,
\Delta_{n,m}[\beta_s] \, \gamma(U^{(n)}_{\kappa(n,s)})\right) {\rm d}s \\
&\quad + \int_{0}^{t}e^{-C_1s}\left( \tau\!\left(\gamma(U^{(n)}_{\kappa(n,s)}) \,\beta(U^{(m)}_{\kappa(m,s)})\right)\,
\Delta_{n,m}[\beta_s]\,\Delta_{n,m}[\gamma_s] \right) {\rm d}s \\
&\quad + \int_0^t e^{-C_1 s} \left( \tau\! \left(\Delta_{n,m}[\gamma_s] \,\Delta_{n,m}[\beta_s]\right)\,
\gamma(U^{(n)}_{\kappa(n,s)})\,\beta(U^{(m)}_{\kappa(m,s)}) \right) {\rm d}s\\
&\quad + \int_0^t e^{-C_1 s} \left( \tau\! \left(\Delta_{n,m}[\gamma_s] \,\beta(U^{(m)}_{\kappa(m,s)})\right)\,
\beta(U^{(m)}_{\kappa(m,s)}) \,\Delta_{n,m}[\gamma_s] \right) {\rm d}s\\
&\quad + M_t^{(n,m)},\\
\end{split}
\end{equation}
where the term $M_t^{(n,m)}$ is given by
\begin{equation*}
\begin{split}
M_t^{(n,m)} &= \int_{0}^{t}e^{-C_1s}\Delta_{n,m} [U_s]
\left( \Delta_{n,m}[\beta_s] \,{\rm d}W_s\,\gamma(U^{(n)}_{\kappa(n,s)}) +\beta(U^{(m)}_{\kappa(m,s)})\,{\rm d}W_s\,\Delta_{n,m}[\gamma_s]\right) \\
&\quad +\int_{0}^{t}e^{-C_1s}\left( \Delta_{n,m}[\beta_s] \,{\rm d}W_s\,\gamma(U^{(n)}_{\kappa(n,s)})+\beta(U^{(m)}_{\kappa(m,s)})\,{\rm d}W_s\,\Delta_{n,m}[\gamma_s]\right) \Delta_{n,m} [U_s].
\end{split}
\end{equation*}
Applying Lemma \ref{lem:marti} then gives that for $t \in [0, T_M^{(n)} \wedge T_M^{(m)}]$
\begin{equation*}
\tau(M_t^{(n,m)}) = 0.
\end{equation*}
Consequently, taking the trace on both sides of \eqref{eq:DU-bound-ito} yields, for all $t$ in this interval,
\begin{equation}\label{eq:DU-bound-estimate}
\begin{split}
e^{-C_1t}&\|\Delta_{n,m} [U_t]\|_2^2 \\
& = -C_1 \int_0^t e^{-C_1s} \|\Delta_{n,m} [U_s]\|_2^2 {\rm d}s +2 \int_{0}^{t}e^{-C_1s}
\langle\Delta_{n,m} [U_s],\Delta_{n,m}[\alpha_s] \rangle {\rm d}s \\
&\quad + \int_{0}^{t}e^{-C_1s}\left( \tau\!\big(\gamma(U^{(n)}_{\kappa(n,s)}) \,\Delta_{n,m}[\beta_s]\big)\right)^2 {\rm d}s\\
&\quad+2 \int_0^t e^{-C_1 s} \tau\!\left(\Delta_{n,m}[\beta_s] \,\Delta_{n,m}[\gamma_s]\right)\cdot
\tau\!\left(\beta(U^{(m)}_{\kappa(m,s)})\,\gamma(U^{(n)}_{\kappa(n,s)})\right) {\rm d}s\\
&\quad + \int_0^t e^{-C_1 s} \left( \tau\! \big( \beta(U^{(m)}_{\kappa(m,s)})\, \Delta_{n,m}[\gamma_s]\big) \right)^2 {\rm d}s.
\end{split}
\end{equation}

Firstly, we decompose 
\begin{equation*}
\langle\Delta_{n,m} [U_s],\Delta_{n,m}[\alpha_s]\rangle
= \bigl\langle \Delta_{n,m} [U_s]+\Delta_{n,m} [p_s],\,\Delta_{n,m}[\alpha_s]\bigr\rangle
- \langle\Delta_{n,m} [p_s],\Delta_{n,m}[\alpha_s]\rangle .
\end{equation*}
Then, the non-commutative Cauchy–Schwarz inequality \cite{pisier_non-commutative_1997} 
\begin{equation}
|\langle X, Y \rangle| \leq \|X\|_2 \cdot \|Y\|_2
\end{equation}
and \eqref{eq:loclip-alpha} imply
\begin{equation*}
\begin{split}
\bigl\langle \Delta_{n,m} [U_s] + \Delta_{n,m} [p_s],\,\Delta_{n,m}[\alpha_s]\bigr\rangle
& \leq \|\Delta_{n,m}[\alpha_s]\|_2 \cdot \|\Delta_{n,m} [U_s] + \Delta_{n,m} [p_s]\|_2 \\
& \le L_\alpha(M)\|\Delta_{n,m} [U_s] + \Delta_{n,m} [p_s]\|_2^2 .
\end{split}
\end{equation*}
Moreover, \eqref{eq:bound-Dalpha} yield
\begin{equation*}
\begin{split}
-\langle\Delta_{n,m} [p_s],\Delta_{n,m}[\alpha_s]\rangle
&\le \|\Delta_{n,m} [p_s]\|_2 \cdot \|\Delta_{n,m}[\alpha_s]\|_2\\
&\le 2K_\alpha(M)\|\Delta_{n,m} [p_s]\|_2.
\end{split}
\end{equation*}
Therefore,
\begin{equation}\label{eq:alpha-est}
\langle\Delta_{n,m} [U_s],\Delta_{n,m}[\alpha_s]\rangle
\le L_\alpha(M)\|\Delta_{n,m} [U_s] + \Delta_{n,m} [p_s] \|_2^2
+ 2K_\alpha(M)\|\Delta_{n,m} [p_s]\|_2.
\end{equation}

Secondly, combining \eqref{eq:beta-gamma-Lip-loc} and \eqref{eq:bound-loc}, we obtain 
\begin{equation}\label{eq:beta-gamma-est}
\begin{split}
&\tau\!\left(\Delta_{n,m}[\beta_s]\,\Delta_{n,m}[\gamma_s] \right) \cdot
\tau\!\left(\beta(U^{(m)}_{\kappa(m,s)})\,\gamma(U^{(n)}_{\kappa(n,s)})\right) \\
&\quad\quad \quad\le \|\Delta_{n,m}[\beta_s] \|_2\|\Delta_{n,m} [\gamma_s] \|_2
\;\|\beta(U^{(m)}_{\kappa(m,s)})\|_2\|\gamma(U^{(n)}_{\kappa(n,s)})\|_2 \\
&\quad \quad \quad \le L_\beta(M)L_\gamma(M)K_\beta(M)K_\gamma(M)\,
\|\Delta_{n,m} [U_s] + \Delta_{n,m} [p_s]\|_2^2.
\end{split}
\end{equation}
Similarly, 
$$\left( \tau\!\big(\gamma(U^{(n)}_{\kappa(n,s)}) \,\Delta_{n,m}[\beta_s]\big)\right)^2 \le L^2_\beta(M)K^2_\gamma(M)\,
\|\Delta_{n,m} [U_s] + \Delta_{n,m} [p_s]\|_2^2$$
 and 
$$\left( \tau\!\big( \beta(U^{(m)}_{\kappa(m,s)})\, \Delta_{n,m}[\gamma_s]\big) \right)^2 \le K^2_\beta(M)L^2_\gamma(M)\,
\|\Delta_{n,m} [U_s] + \Delta_{n,m} [p_s]\|_2^2.$$

Then, inserting \eqref{eq:alpha-est} and \eqref{eq:beta-gamma-est} into \eqref{eq:DU-bound-estimate} gives
\begin{equation*}\label{eq:DU-bound-norm}
\begin{split}
e^{-C_1t}&\|\Delta_{n,m} [U_t]\|_2^2 \\
& \leq -C_1 \int_0^t e^{-C_1s} \|\Delta_{n,m} [U_s]\|_2^2 {\rm d}s\\
& \quad +2 \int_{0}^{t}e^{-C_1s}
\left( L_\alpha(M)\|\Delta_{n,m} [U_s] + \Delta_{n,m} [p_s] \|_2^2
+ 2K_\alpha(M)\|\Delta_{n,m} [p_s]\|_2 \right) {\rm d}s \\
&\quad+(L_\beta(M)K_\gamma(M)+L_\gamma(M)K_\beta(M))^2 \int_0^t e^{-C_1 s} 
\|\Delta_{n,m} [U_s] + \Delta_{n,m} [p_s]\|_2^2 {\rm d}s.
\end{split}
\end{equation*}
Now choose
\begin{equation*}
C_1 := 4L_\alpha(M) + 2(L_\beta(M)K_\gamma(M)+L_\gamma(M)K_\beta(M))^2,
\end{equation*}
and using the elementary inequality 
\begin{equation*}
\|\Delta_{n,m} [U_s] + \Delta_{n,m} [p_s]\|_2^2
\le 2\|\Delta_{n,m} [U_s]\|_2^2 + 2\|\Delta_{n,m} [p_s]\|_2^2,
\end{equation*}
we obtain
\begin{equation*}
e^{-C_1t}\|\Delta_{n,m} [U_t]\|_2^2
\le \int_{0}^{t}e^{-C_1s}\left( 4K_\alpha(M)\|\Delta_{n,m} [p_s]\|_2
+ C_1\|\Delta_{n,m} [p_s]\|_2^2\right) {\rm d}s.
\end{equation*}

Finally, note that $\sup_{t\in[0,T_M^{(n)}]}\|p^{(n)}_t\|_2\to 0$ as $n \to\infty$ (see \eqref{eq:p-bounds}). We have 
\begin{equation*}
\sup_{t\in[0,T_M^{(n)}\wedge T_M^{(m)}]}\|\Delta_{n,m} [p_t]\|_2 \to 0
\end{equation*}
as $n, m\to\infty,$ which completes the proof. 
\end{proof}

\begin{proof}[Proof of Proposition~\ref{prop:Cauchy-prop}]
By Lemma \ref{lem:Cauchy-prop-stopping-time}, it therefore suffices to prove that 
\begin{equation*}
\lim_{M\to\infty}\liminf_{n\to\infty} T^{(n)}_M= \infty.
\end{equation*}
Let $C_2$ be a constant that will be determined later, the free It\^o formula \eqref{eq:free-ito-polynomial} implies
\begin{equation}\label{eq:U-bound-ito}
		\begin{split}
		e^{-C_2t} (U_t^{(n)})^2 &=U^2_0- C_2 \int_0^t e^{-C_2 s}(U_s^{(n)})^2 {\rm d}s \\
		&\quad + \int_0^t e^{-C_2s} \left( U_s^{(n)}  \alpha(U^{(n)}_{\kappa(n,s)})  {\rm d}s + \alpha(U^{(n)}_{\kappa(n,s)})  U_s^{(n)}\right) {\rm d}s \\
		&\quad+\int_0^t e^{-C_2 s} \tau \left(\gamma(U^{(n)}_{\kappa(n,s)}) \beta(U^{(n)}_{\kappa(n,s)}) \right) \beta(U^{(n)}_{\kappa(n,s)})\gamma(U^{(n)}_{\kappa(n,s)}) {\rm d}s \\
		&\quad +M_t^{(n)},\\
		\end{split}
		\end{equation}
where 
\begin{equation*}
\begin{split}
M_t^{(n)}&:=\int_0^t e^{-C_2s} \left( U_s^{(n)}  \beta(U^{(n)}_{\kappa(n,s)})  {\rm d}W_s \gamma(U^{(n)}_{\kappa(n,s)}) + \beta(U^{(n)}_{\kappa(n,s)}) {\rm d}W_s \gamma(U^{(n)}_{\kappa(n,s)}) U_s^{(n)}\right).
\end{split}
\end{equation*}

For $t\in[0,T_M^{(n)}],$ Lemma \ref{lem:marti} implies $\tau(M_t^{(n)})=0.$ Taking the trace on both sides of \eqref{eq:U-bound-ito}, we obtain
\begin{equation*}
\begin{split}
		e^{-C_2t} \| U_t^{(n)}\|_2^2 &=\| U_0\|_2^2 - C_2 \int_0^t e^{-C_2 s} \| U_s^{(n)}\|_2^2 {\rm d}s \\
		&\quad +2 \int_0^t e^{-C_2s} \left\langle U_s^{(n)} ,  \alpha(U^{(n)}_{\kappa(n,s)}) \right\rangle  {\rm d}s \\
		&\quad+\int_0^t e^{-C_2 s} \left[  \tau \left(\gamma(U^{(n)}_{\kappa(n,s)}) \beta(U^{(n)}_{\kappa(n,s)}) \right) \right]^2 {\rm d}s.
		\end{split}
\end{equation*}

By Assumption \ref{assumption:ass2} and a similar argument of \eqref{eq:alpha-est}, we have for any $t \in [0, T^{(n)}_M]$,
\begin{equation*}\label{eq:U-bound-norm}
\begin{split}
e^{-C_2 t} \|U_t^{(n)}\|_2^2 &\le \|U_0\|_2^2- C_2 \int_0^t e^{-C_2 s} \| U_s^{(n)}\|_2^2 {\rm d}s \\
& \quad + \int_0^t e^{-C_2 s} \left( 2K_{\alpha}(M) \|p_s^{(n)}\|_2 + L_1 \|U^{(n)}_{\kappa(n,s)}\|_2^2 + L_2  \right) {\rm d}s.
\end{split}
\end{equation*}
Using the fact that $U^{(n)}_{\kappa(n,s)}=U_s^{(n)}+p_s^{(n)}$ and the elementary inequality 
\begin{equation*}
\|U_s^{(n)}+p_s^{(n)}\|_2^2 \le 2\|U_s^{(n)}\|_2^2 + 2\|p_s^{(n)}\|_2^2
\end{equation*}
and choosing $C_2 := 2L_1$, we obtain
\begin{equation}\label{eq:U-bound-3}
\begin{split}
e^{-2L_1t} \|U_t^{(n)}\|_2^2 &\leq \|U_0\|_2^2 + \int_0^t e^{-2L_1s} \left(2K_\alpha(M)\|p_s^{(n)}\|_2 + 2L_1\|p_s^{(n)}\|_2^2 + L_2 \right) {\rm d}s.
\end{split}
\end{equation}
Note that $\|U^{(n)}_{T_M^{(n)}}\|_2 \ge M/3$ and 
\begin{equation*}
\sup_{t\in[0,T_M^{(n)}]}\|p_t^{(n)}\|_2 \to 0
\end{equation*} 
as $n \to \infty$. By taking $t = T_M^{(n)}$ and letting $n\to\infty$ in \eqref{eq:U-bound-3}, we derive that
\begin{equation}\label{eq:limit-ineq}
\frac{M^2}{9} \leq e^{2L_1 T_M} \|U_0\|_2^2 + \frac{L_2}{2L_1}(e^{2L_1 T_M} - 1) \quad (L_1 \neq 0),
\end{equation}
where $T_M := \liminf_{n\to\infty} T_M^{(n)}$. For $L_1 = 0$, the inequality becomes
\begin{equation*}
\frac{M^2}{9} \leq \|U_0\|_2^2 + L_2 T_M.
\end{equation*}

Hence, the above inequalities induces the following behavior of the explosion time:
\begin{itemize}
\item If $L_1 > 0$, then as $M\to\infty$, the right-hand side of \eqref{eq:limit-ineq} grows exponentially in $T_M$, forcing $T_M \to \infty$.
\item If $L_1 = 0$, then $T_M \geq \frac{M^2/9 - \|U_0\|_2^2}{L_2} \to \infty$ as $M\to\infty$ (provided $L_2 > 0$).
\item If $L_1 < 0$, the factor $e^{2L_1T_M}$ decays, and the right-hand side of \eqref{eq:limit-ineq} is bounded above by $\|U_0\|_2^2 + \frac{L_2}{2|L_1|}$.
Hence, for sufficiently large $M$, inequality \eqref{eq:limit-ineq}  cannot hold for any finite $T_M$. This implies that $T_M$ must be infinite, meaning the solution does not reach norm $M/3$ in finite time; equivalently, $T_M \to \infty$ as $M \to \infty$ (in the sense that the explosion time is infinite for all large $M$).
\end{itemize}
Thus, in all three cases, $\lim_{M\to\infty}T_M = \infty$, which completes our proof.
\end{proof}


\subsection{Existence and uniqueness of the solution} 

By Proposition \ref{prop:Cauchy-prop} and the completeness of $L_2(\mathscr{A},\tau)$, there exists a process $U_t \in L_2(\mathscr{A},\tau)$, such that for any given $T>0$
\begin{equation}\label{eq:converge}
\sup_{t\in[0,T]}\|U^{(n)}_t-U_t\|_2\to 0 \quad \text{as } n \to \infty.
\end{equation}
We will check $U_t$ fulfills the free SDE \eqref{eq:fSDE}. 

\begin{proposition}\label{prop:solution}
Assume that Assumptions \ref{assumption:ass1} and \ref{assumption:ass2} hold. Then for any given $T>0$ the process $U_t$ given by \eqref{eq:converge} is a unique solution to the free SDE \eqref{eq:fSDE}, i.e., 
\begin{equation}\label{eq:solution-fSDE}
U_t = U_0 + \int_0^t \alpha(U_s) {\rm d}s + \int_0^t \beta(U_s) {\rm d}W_s \gamma (U_s)
\end{equation}
for $t \in [0, T].$ Moreover,
\begin{equation*}
\sup_{t \in [0, T]} \|U_t\|_2 < \infty. 
\end{equation*}
\end{proposition}

\begin{proof}
For $M\in[0,\infty)$ define the stopping time 
\begin{equation*}
\widetilde{T}_M:=\inf\{t\in[0,T]: \sup_{n\in\mathbb{N}}\|U^{(n)}_{\kappa(n,t)}\|_2\vee \|U_t\|_2>M\}\wedge T.
\end{equation*} 

Hence, the local Lipschitz property of $\alpha$ implies
\begin{equation*}\label{eq:alpha-conv}
\begin{split}
\biggl\|\int_{0}^{t \wedge \widetilde{T}_M} \alpha\bigl(U^{(n)}_{\kappa(n,s)}\bigr) \, {\rm d}s 
- \int_{0}^{t \wedge \widetilde{T}_M} \alpha(U_s) \, {\rm d}s \biggr\|_2 
&\leq \int_{0}^{t \wedge \widetilde{T}_M} 
\bigl\| \alpha\bigl(U^{(n)}_{\kappa(n,s)}\bigr) - \alpha(U_s) \bigr\|_2 \, {\rm d}s \\
&\longrightarrow 0 \quad \text{as } n \to \infty.
\end{split}
\end{equation*}
Moreover, the free It\^o isometry \eqref{eq:isometry} yields 
\begin{equation*}\label{eq:beta-gamma-conv}
\begin{split}
\biggl\| \int_{0}^{t \wedge \widetilde{T}_M}& \beta\bigl(U^{(n)}_{\kappa(n,s)}\bigr) \, {\rm d}W_s \, \gamma\bigl(U^{(n)}_{\kappa(n,s)}\bigr)
- \int_{0}^{t \wedge \widetilde{T}_M} \beta(U_s) \, {\rm d}W_s \, \gamma(U_s) \biggr\|_2 \\
&= \biggl\| \int_{0}^{t \wedge \widetilde{T}_M} \bigl(\beta\bigl(U^{(n)}_{\kappa(n,s)}\bigr) - \beta(U_s)\bigr) \, {\rm d}W_s \, \gamma\bigl(U^{(n)}_{\kappa(n,s)}\bigr) \\
&\qquad + \int_{0}^{t \wedge \widetilde{T}_M} \beta(U_s) \, {\rm d}W_s \, \bigl(\gamma\bigl(U^{(n)}_{\kappa(n,s)}\bigr) - \gamma(U_s)\bigr) \biggr\|_2 \\
&\le \biggl( \int_{0}^{t \wedge \widetilde{T}_M} \bigl\| \beta\bigl(U^{(n)}_{\kappa(n,s)}\bigr) - \beta(U_s) \bigr\|_2^2 \, 
\bigl\| \gamma\bigl(U^{(n)}_{\kappa(n,s)}\bigr) \bigr\|_2^2 \, {\rm d}s \biggr)^{\frac12} \\
&\qquad + \biggl( \int_{0}^{t \wedge \widetilde{T}_M} \bigl\| \gamma\bigl(U^{(n)}_{\kappa(n,s)}\bigr) - \gamma(U_s) \bigr\|_2^2 \,
\bigl\| \beta(U_s) \bigr\|_2^2 \, {\rm d}s \biggr)^{\frac12} \\
&\longrightarrow 0 \quad \text{as } n \to \infty.
\end{split}
\end{equation*}

Since $\lim_{M\to\infty}\liminf_{n\to\infty} T^{(n)}_M= \infty$, there exists $\widetilde{M}\in[0,\infty)$ such that $\widetilde{T}_{\widetilde{M}}=T.$ So, $U_t$ is a solution to the free SDE \eqref{eq:solution-fSDE} for all $t \in [0, T].$

Now for the uniformly boundedness of $U_t,$ it follows from \eqref{eq:converge} and \eqref{eq:limit-ineq} that for any given $T>0$
\begin{equation*}\label{eq:converge}
\sup_{t\in[0,T]}\|U_t\|_2^2\leq \sup_{t\in[0,T]}\|U^{(n)}_t\|_2^2 +\sup_{t\in[0,T]} \|U^{(n)}_t-U_t\|_2^2\leq e^{2L_1 T} \|U_0\|_2^2 + \frac{L_2}{2L_1}(e^{2L_1 T} - 1)
\end{equation*}
for $L_1\neq 0$ and $n \to \infty$. Similarly, for $L_1=0,$ we have $\sup_{t\in[0,T]}\|U_t\|_2^2\leq\|U_0\|_2^2+L_2T.$

Next, suppose that $U_t$ and $\widetilde{U}_t$ are two solutions to \eqref{eq:fSDE} with the same initial value, and that both are uniformly bounded by $M$.
By \eqref{eq:solution-fSDE}, we have for any $t\in[0,T]$
\begin{equation*}
\begin{split}
    U_t - \widetilde{U}_t &= \int_0^t \bigl(\alpha(U_s) - \alpha(\widetilde{U}_s)\bigr) \, {\rm d}s 
    + \int_0^t \bigl(\beta(U_s) - \beta(\widetilde{U}_s)\bigr) \, {\rm d}W_s \, \gamma(U_s) \\
    &\quad + \int_0^t \beta(\widetilde{U}_s) \, {\rm d}W_s \, \bigl(\gamma(U_s) - \gamma(\widetilde{U}_s)\bigr).
\end{split}
\end{equation*}
Using Assumption~\ref{assumption:ass1}, \eqref{eq:bound-loc}, and the free It\^o isometry \eqref{eq:isometry}, we obtain
\begin{equation*}
\begin{split}
    \|U_t - \widetilde{U}_t\|_2 
    &\leq \int_0^t L_{\alpha}(M) \, \|U_s - \widetilde{U}_s\|_2 \, {\rm d}s \\
    &\quad + \biggl( \int_0^t \bigl(K^2_{\gamma}(M) L^2_{\beta}(M) + K^2_{\beta}(M) L^2_{\gamma}(M)\bigr) 
    \|U_s - \widetilde{U}_s\|_2^2 \, {\rm d}s \biggr)^{\frac{1}{2}} \\
    &\leq C_3 \biggl( \int_0^t \|U_s - \widetilde{U}_s\|_2^2 \, {\rm d}s \biggr)^{\frac{1}{2}},
\end{split}
\end{equation*}
where $C_3 = L_{\alpha}(M)\sqrt{T} + \bigl(K^2_{\gamma}(M) L^2_{\beta}(M) + K^2_{\beta}(M) L^2_{\gamma}(M)\bigr)^{\frac{1}{2}}$, 
and we have applied the Cauchy--Schwarz inequality to obtain the second estimate. 
By Gronwall's inequality \cite{oksendal2003}, it follows that $\|U_t - \widetilde{U}_t\|_2^2 = 0$ for all $t \in [0, T]$. 
Hence $U_t = \widetilde{U}_t$, which establishes the uniqueness of the solution.

It remains to prove that $t\to U_t$ is continuous in $L^2(\mathscr{A},\tau)$. Indeed,
since for any $t\geq s\geq 0,$
\begin{equation*}
U_t-U_s =  \int_s^t \alpha(U_r) {\rm d}r + \int_s^t \beta(U_r) {\rm d}W_r \gamma (U_r),
\end{equation*}
we obtain that there exists a constant $K>0$ such that
\begin{equation*}
\begin{split}
\|U_t-U_s\|_2 &\leq \|\int_s^t \alpha(U_r) {\rm d}r\|_2 + \|\int_s^t \beta(U_r) {\rm d}W_r \gamma (U_r)\|_2\\
&\leq  \int_s^t \| \alpha(U_r)\|_2 {\rm d}r + \left(\int_s^t \|\beta(U_r)\|_2^2 \| \gamma (U_r)\|_2^2 {\rm d}r \right)^{\frac{1}{2}}\\
&\leq K(t-s)^{\frac{1}{2}},
\end{split}
\end{equation*}
where we have used the free It\^o isometry \eqref{eq:isometry}, uniformly boundedness of $U_t$ and \eqref{eq:bound-loc}.
\end{proof}

\subsection{Weak uniqueness of the equation}\label{subsec:weak-unique}
We recall the notion of the Cauchy transform. 
\begin{definition}[\cite{mingoFreeProbabilityRandom2017}]
Let $\mu$ be a probability measure on $\mathbb{R}.$ Its Cauchy transform is given by 
\begin{equation}\label{eq:Cauchy-transform}
G_\mu (z) : = \int_\mathbb{R} \frac{1}{z-t} \, \mu ({\rm d}t)
\end{equation}
for any $z \in \mathbb{C} \backslash \mathbb{R}.$ 
\end{definition}
It was known that the Cauchy transform $G_{\mu}$ is analytic on $\mathbb{C}^+$ and uniquely determines the distribution $\mu$ via the following Stieltjes inversion formula:
\begin{equation}\label{eq:stiel-inver}
\mu ((a,b)) = -\frac{1}{\pi} \lim_{\varepsilon \to 0^+} \int_a^b \Im\, G_\mu (x + i\varepsilon) \, {\rm d}x
\end{equation}
for continuous points $a < b$ of $\mu$. Moreover, 

\begin{proposition}[{\cite{mingoFreeProbabilityRandom2017}}]\label{prop:Cauchy-transform}
If $\{\mu_n\}_{n\in\mathbb N}$ is a sequence of probability measures on $\mathbb R$ such that
\begin{equation*}
G_{\mu_n}(z)\to G_\mu(z), \qquad z\in\mathbb C^{+},
\end{equation*}
for some probability measure $\mu$ on $\mathbb R$, then $\mu_n$ converges weakly to $\mu$.
\end{proposition}

For $U \in \widetilde{\mathscr{A}}_{\mathrm{\mathrm{sa}}}$, the \emph{resolvent} map $R_U(z) = (z - U)^{-1}$ is an analytic map from $\mathbb{C}^+$ to $\mathscr{A}$. Furthermore,
we have $\|R_U(z)\| \leq 1/ |{\Im (z)}|$ for any $z \in \mathbb{C}^+.$
Let $\mu_U$ be the distribution of $U,$ by the functional calculus, we have 
\begin{equation*}
G_{\mu_U} (z) = \tau (R_U(z))
\end{equation*}
for any $z \in \mathbb{C}^+.$

\begin{proposition}\label{prop:same-law}
Assume that Assumptions \ref{assumption:ass1} and \ref{assumption:ass2} hold. Then the free SDE \eqref{eq:fSDE} admits \textit{weak uniqueness}.
\end{proposition}

\begin{proof}
Let $\{U_t, W_t\}_{t\geq 0}$ and $\{\widehat{U}_t, \widehat{W}_t\}_{t \geq 0}$ be two weak solutions to the free SDE \eqref{eq:fSDE}, and suppose that $U_0$ and $\widehat{U}_0$ have the same distribution.

Denote by $\{U_t^{(n)}\}_{t\ge0}$ and $\{\widehat{U}_t^{(n)}\}_{t\ge0}$ the Euler approximations \eqref{eq:Euler-approx} 
corresponding to $(U_0, W_t)$ and $(\widehat{U}_0, \widehat{W}_t)$, respectively.
We first show that for every $n \in \mathbb{N}$ and $t \ge 0$, $U_t^{(n)}$ and $\widehat{U}_t^{(n)}$ have the same distribution.
Since they are bounded self-adjoint operators, their distributions are determined by polynomial moments \cite[Remark 1.9]{NS2006}; 
it therefore suffices to prove that for every polynomial $p$,
\begin{equation*}
    \tau\bigl(p(U_t^{(n)})\bigr) = \tau\bigl(p(\widehat{U}_t^{(n)})\bigr), \qquad t \ge 0.
\end{equation*}

We argue by induction over the mesh intervals.
For the first interval, let $0<s\le \frac1n\wedge t$. Then
\begin{equation*}
U_s^{(n)}
=
U_0+s\,\alpha(U_0)+\beta(U_0)\,W_s\,\gamma(U_0),
\end{equation*}
and
\begin{equation*}
\widehat U_s^{(n)}
=
\widehat U_0+s\,\alpha(\widehat U_0)
+\beta(\widehat U_0)\,\widehat W_s\,\gamma(\widehat U_0).
\end{equation*}
Since $\alpha,\beta,\gamma$ do not break causality, we have $\alpha(U_0),\beta(U_0),\gamma(U_0)\in W^*(U_0).$ Then $\tau\bigl(p(U_s^{(n)})\bigr)$ can be expressed as a linear combination of products of the form
\begin{equation*}
\tau(p_1W_sp_2W_s\cdots p_{d-1}W_sp_d),
\end{equation*}
where each $p_i$ is a polynomial in $U_0,\alpha(U_0),\beta(U_0)$ and $\gamma(U_0)$.
Moreover, $W_s$ is free from $W^*(U_0)$. Then by the free moment cumulant formula \cite{NS2006}, $\tau(p_1W_sp_2W_s\cdots p_{d-1}W_sp_d)$ can be expressed as a linear combination of products of the form
\begin{equation*}
\Pi_{r=1}^{d'}(\tau(\Pi_{i\in I_r}p_i)) \Pi_{m=1}^v \tau(W_s^{l_m}),
\end{equation*}
where $\{1,\ldots,d\}=I_1\cup\cdots\cup I_{d'},$ $d=l_1+\cdots+l_v.$ 

Similarly, using the fact that $\widehat W_s$ is free from $W^*(\widehat{U}_0),$
$\tau\bigl(p(\widehat U_s^{(n)})\bigr)$ can be expressed as a linear combination of products of the form
\begin{equation*}
\Pi_{r=1}^{d'}(\tau(\Pi_{i\in I_r}\widehat p_i)) \Pi_{m=1}^v \tau(\widehat W_s^{l_m}),
\end{equation*}
where each $\widehat{p}_i = p_i (\widehat U_0,\alpha(\widehat U_0),\beta(\widehat U_0), \gamma(\widehat U_0))$.

Since $U_0$ and $\widehat U_0$ have the same distribution, it follows that for any polynomial $p$,
\begin{equation*}
\tau\bigl(p(U_s^{(n)})\bigr)
=
\tau\bigl(p(\widehat U_s^{(n)})\bigr),
\qquad 0<s\le \frac1n\wedge t.
\end{equation*}

Now fix $k\in\mathbb N$ and assume that $U_{k/n}^{(n)}$ and $\widehat U_{k/n}^{(n)}$ have the same distribution. Let $s\in \Bigl(\frac{k}{n},\frac{k+1}{n}\Bigr]\cap[0,t]$ and set
$\Delta W_{k,s}:=W_s-W_{k/n},~\Delta \widehat W_{k,s}:=\widehat W_s-\widehat W_{k/n}.$
Then
\begin{equation*}
U_s^{(n)}
=
U_{k/n}^{(n)}
+\Bigl(s-\frac{k}{n}\Bigr)\alpha\bigl(U_{k/n}^{(n)}\bigr)
+\beta\bigl(U_{k/n}^{(n)}\bigr)\,\Delta W_{k,s}\,\gamma\bigl(U_{k/n}^{(n)}\bigr),
\end{equation*}
and
\begin{equation*}
\widehat U_s^{(n)}
=
\widehat U_{k/n}^{(n)}
+\Bigl(s-\frac{k}{n}\Bigr)\alpha\bigl(\widehat U_{k/n}^{(n)}\bigr)
+\beta\bigl(\widehat U_{k/n}^{(n)}\bigr)\,\Delta \widehat W_{k,s}\,\gamma\bigl(\widehat U_{k/n}^{(n)}\bigr).
\end{equation*}

Note that $\Delta W_{k,s}$ and $\Delta \widehat W_{k,s}$ have the same semicircular distribution with variance $s-\frac{k}{n}$, and $\Delta W_{k,s}$ is free from $\mathscr A_{ \leq k/n}$ and $\Delta \widehat W_{k,s}$ is free from $\widehat{\mathscr A}_{ \leq k/n}:= W^*(\widehat{U}_0, \widehat{W}_r:r\le k/n).$
Therefore, by the induction hypothesis and the assumption on $\alpha,\beta,\gamma$, the same procedure as above yields that for every polynomial $p$,
\begin{equation*}
\tau\bigl(p(U_s^{(n)})\bigr)
=
\tau\bigl(p(\widehat U_s^{(n)})\bigr).
\end{equation*}
This proves that $U_t^{(n)}$ and $\widehat U_t^{(n)}$ have the same distribution for every $n\in\mathbb N$ and every $t\ge0$.

It remains to pass to the limit as $n\to\infty$. Fix $t\ge0$ and $z\in\mathbb C^+$. Since $U_t^{(n)}\to U_t $ in $L^2(\mathscr A,\tau),$
the resolvent identity yields
\begin{equation*}
(z-U_t^{(n)})^{-1}-(z-U_t)^{-1}
=
(z-U_t^{(n)})^{-1}(U_t^{(n)}-U_t)(z-U_t)^{-1},
\end{equation*}
and hence
\begin{equation}\label{eq:cauchy-conv}
\bigl\|(z-U_t^{(n)})^{-1}-(z-U_t)^{-1}\bigr\|_2
\le
\frac{1}{(\Im z)^2}\,\|U_t^{(n)}-U_t\|_2 \to 0.
\end{equation}
Therefore,
\begin{equation*}
\tau\bigl((z-U_t^{(n)})^{-1}\bigr)\to \tau\bigl((z-U_t)^{-1}\bigr).
\end{equation*}
Similarly, since $\widehat U_t^{(n)}\to \widehat U_t$ in $L^2(\mathscr A,\tau),$ we deduce
\begin{equation*}
\tau\bigl((z-\widehat U_t^{(n)})^{-1}\bigr)\to
\tau\bigl((z-\widehat U_t)^{-1}\bigr).
\end{equation*}
Since $U_t^{(n)}$ and $\widehat U_t^{(n)}$ have the same distribution for every $n$, we have
\begin{equation*}
\tau\bigl((z-U_t^{(n)})^{-1}\bigr)
=
\tau\bigl((z-\widehat U_t^{(n)})^{-1}\bigr),
\qquad z\in\mathbb C^+.
\end{equation*}
Passing to the limit gives
\begin{equation*}
\tau\bigl((z-U_t)^{-1}\bigr)
=
\tau\bigl((z-\widehat U_t)^{-1}\bigr),
\qquad z\in\mathbb C^+.
\end{equation*}
Hence $U_t$ and $\widehat U_t$ have the same Cauchy transform, and therefore the same distribution. 
\end{proof}

\begin{remark}\label{rmk:weak-unique-free-diffusion}
For the free diffusion equation \eqref{eq:fDiff} such that the drift coefficient $\alpha$ satisfies \eqref{eq:local-operator-Lip} and \eqref{eq:speicher-biane}, Assumption \ref{assumption:ass1} may not hold. However, Biane and Speicher showed there is a Picard iteration $U_t^{(n)}$ given by \eqref{eq:picard-iter}, such that $\|U_t^{(n)} - U_t\| \to 0$ as $n \to \infty$ (and a similar statement holds for $\widehat{U}_t$). Moreover, by a similar argument as above, we may show that  
$\tau\bigl(p(U_t^{(n)})\bigr) =\tau\bigl(p(\widehat U_t^{(n)})\bigr)$ for any polynomial $p$, and thus establish the weak uniqueness of the free diffusion equation \eqref{eq:fDiff}. 
\end{remark}

\begin{remark}\label{rmk:weak-unique}
For the equation \eqref{eq:fSDE-one-side}, assume that the coefficients satisfy the local operator Lipschitz condition. Capitaine and Donati‑Martin \cite[Proposition A.1]{capitaine_free_2005} also constructed Picard iterates converging to the solution. Hence, by an argument similar to Remark~\ref{rmk:weak-unique-free-diffusion}, weak uniqueness also holds for the equation \eqref{eq:fSDE-one-side}.
\end{remark}


\subsection{Free Markov property}\label{subsec:free-markov-process}
Recall that 
for each $t\ge0,$ 
\begin{equation*}
\mathscr A_{\le t}=W^*(U_0,W_r:r\le t),\quad
\mathscr A_{\ge t}=W^*(U_t,W_r-W_t:r\ge t),\quad
\mathscr A_{=t}=W^*(U_t).
\end{equation*}
And
\begin{equation*}
\mathscr B_{\le t}=W^*(U_r:r\le t),
\qquad
\mathscr B_{\ge t}=W^*(U_r:r\ge t).
\end{equation*}

\begin{proof}[Proof of Proposition~\ref{prop:free-Markov}]
By Theorem \ref{thm:solution}, there exists a unique solution $\{U_t\}_{t\ge0}$ to \eqref{eq:fSDE}. 
Fix $t\ge0$, we shall prove that
\begin{equation*}
\mathscr B_{\le t}\subseteq \mathscr A_{\le t},
\qquad
\mathscr B_{\ge t}\subseteq \mathscr A_{\ge t}.
\end{equation*}

We first show that $\mathscr B_{\le t}\subseteq \mathscr A_{\le t}$.
Let $\{U_s^{(n)}\}_{s\ge0}$ be the Euler approximation given by \eqref{eq:Euler-approx}. Since $U_s^{(n)}\to U_s$ in $L^2(\mathscr A,\tau)$, it follows from Remark~\ref{UinA} that $U_s\in L^2(\mathscr A_{\le t},\tau), 0\le s\le t.$
Thus,
\begin{equation*}
\mathscr B_{\le t}=W^*(U_s:0\le s\le t)\subseteq \mathscr A_{\le t}.
\end{equation*}

Next we show that $\mathscr B_{\ge t}\subseteq \mathscr A_{\ge t}$.
For $s\ge t$, the solution satisfies
\begin{equation*}
U_s
=
U_t+\int_t^s \alpha(U_r)\, {\rm d}r+\int_t^s \beta(U_r)\,{\rm d}W_r\,\gamma(U_r).
\end{equation*}
Define the shifted free Brownian motion
\begin{equation*}
\widetilde W_r:=W_{t+r}-W_t,\qquad r\ge0.
\end{equation*}
Since $\{W_t\}_{t\ge0}$ is a free Brownian motion, the increment process
$\{\widetilde W_r\}_{r\ge0}$ is again a free Brownian motion. Then the above equation becomes
\begin{equation*}
U_{t+r}
=
U_t+\int_0^r \alpha(U_{t+h})\, {\rm d}h+\int_0^r \beta(U_{t+h})\,{\rm d}\widetilde W_h\,\gamma(U_{t+h}),
\qquad r\ge0.
\end{equation*}

Applying the same argument as Remark~\ref{UinA} to this shifted equation with initial value $U_t$ and driving noise $\widetilde W$, we obtain $U_{t+r}\in L^2(\mathscr{A}_{\ge t},\tau)$ for $r\ge0$, where we have used the fact that $W^*(U_t,\widetilde W_h:h\le r)\subseteq \mathscr{A}_{\ge t}$.
Therefore
\begin{equation*}
\mathscr B_{\ge t}=W^*(U_s:s\ge t)\subseteq \mathscr A_{\ge t}.
\end{equation*}

Finally, since $\{W_t\}_{t\ge0}$ is a free Brownian motion and $U_0$ is free from it, the algebra $W^*(W_h-W_t:h\ge t)$
is free from $\mathscr A_{\le t}$, and moreover,
\begin{equation*}
\mathscr A_{=t}=W^*(U_t)=\mathscr B_{=t}\subseteq \mathscr A_{\le t}.
\end{equation*}
Therefore, by \cite[Lemma 2.1]{bianeFreeDiffusionsFree2001}, the subalgebras $\mathscr A_{\le t}$ and $\mathscr A_{\ge t}$ are $\mathscr A_{=t}$-free.
Since $\mathscr B_{\le t}\subseteq \mathscr A_{\le t},\mathscr B_{\ge t}\subseteq \mathscr A_{\ge t},$ it follows that $\mathscr B_{\le t}$ and $\mathscr B_{\ge t}$ are $\mathscr B_{=t}$-free.
Hence, we conclude that $\{U_t\}_{t\ge0}$ is a free Markov process.
\end{proof}

\section{Stationary distribution of free SDEs}\label{sec:stat}

\subsection{Existence and uniqueness of the stationary distribution}
In this section, we still set $l = 1$ in \eqref{eq:fSDE} for simplicity. By our assumptions, for any given free Brownian motion $\{W_t\}_{t\geq 0}$ and for any initial value $u$ (at time $0$) that is free from $\{W_t\}_{t\geq 0}$, there exists a unique solution to the free SDE. We denote this solution by $U_t^u$ to emphasize its dependence on the initial value, and it satisfies $\sup_{t\geq 0} \|U_t^u\| < \infty.$ 

\begin{notation}
Let $u,v\in \mathscr A_{sa},$ for any $0\le t< \infty,$ we denote
\begin{itemize}
\item $\Delta_{u,v} [U_t]:=U^u_{t}-U^v_{t};$  
\item $\Delta_{u,v}[\alpha_t]:=\alpha (U^u_{t})-\alpha (U^v_{t});$
\item $\Delta_{u,v}[\beta_t]:=\beta (U^u_{t})-\beta (U^v_{t});$
\item $\Delta_{u,v}[\gamma_t]:=\gamma (U^u_{t})-\gamma (U^v_{t}).$
\end{itemize}
\end{notation}
By \eqref{eq:alpha-onesideLip-loc} in Assumption \hyperlink{assumption:ass3}{(B-1)} , we have 
\begin{equation}\label{eq:difference-s-1}
\tau\left(\Delta_{u,v} [U_t] \cdot \Delta_{u,v} [\alpha_t] \right) \leq L'_\alpha \|\Delta_{u,v} [U_t]\|_2^2.
\end{equation}
Moreover, 
\begin{equation}\label{eq:difference-s-2}
\| \Delta_{u,v} [\beta_t] \|_2 \leq L_\beta \|\Delta_{u,v} [U_t]\|_2 \; \text{and} \; \| \Delta_{u,v} [\gamma_t] \|_2 \leq L_\gamma \|\Delta_{u,v} [U_t]\|_2.
\end{equation}
Meanwhile, it follows from \eqref{eq:bound-loc} that for any $x \in \{u, v\}$ and $t\ge0,$ 
\begin{equation}\label{eq:bound-s}
 \| \beta(U^x_{t}) \|_2 \leq K_\beta \; \text{and} \; \| \gamma(U^x_{t}) \|_2 \leq K_\gamma.
\end{equation}

With the above notations, we have the following propositions, which play important roles in the proof of Theorem~\ref{thm:stationary}.

\begin{proposition}\label{prop:Stability-prop-1}
Assume the hypotheses of Theorem~\ref{thm:stationary}. For any given free Brownian motion $\{W_t\}_{t\geq 0}$,  let $u, v \in \mathscr{A}_{sa}$ that are free from $\{W_t\}_{t\geq 0}.$
Let $\{U^u_{t}\}_{t \geq 0},\{U^v_{t}\}_{t \geq 0}$ be the unique solutions to the free SDE \eqref{eq:fSDE} with initial values $u,v$, respectively. Then there exists a constant $C>0$ such that
\begin{equation}\label{eq:conti}
 \left\| U^u_t - U^v_t \right\|_2 \leq e^{-Ct}\|u-v\|_2
\end{equation}
for any $t \geq 0.$
\end{proposition}

\begin{proof}
By \eqref{eq:solution-fSDE}, we have for any $t \geq 0$
\begin{equation*}
\begin{split}
\Delta_{u,v} [U_t] =& \Delta_{u,v} [U_{0}]+\int_{0}^t \Delta_{u,v}[\alpha_s] {\rm d}s \\
& \quad +\int_{0}^t \Delta_{u,v}[\beta_s]  {\rm d} W_s \gamma(U^u_{s}) +\int_{0}^t\beta (U^v_{s}) {\rm d} W_s \Delta_{u,v}[\gamma_s].
\end{split}
\end{equation*}
Then, for a constant $C_4$, the free It\^{o} formula readily implies that 
\begin{equation*}\label{eq:ito}
		\begin{split}
		d & \left[ e^{C_4 t} (\Delta_{u,v} [U_t])^2 \right]\\
		&=e^{C_4 t} \left( \Delta_{u,v} [U_t] \cdot \Delta_{u,v}[\alpha_t]  {\rm d}t+\Delta_{u,v}[\alpha_t] \cdot \Delta_{u,v} [U_t] {\rm d}t+ C_4 (\Delta_{u,v} [U_t])^2  {\rm d} t  \right) \\
		& \quad 
			+ e^{C_4 t} \tau \left(\gamma(U^u_{t})\cdot\Delta_{u,v}[\beta_t]\right)\Delta_{u,v}[\beta_t] \cdot \gamma(U^u_{t}) {\rm d}t  \\
		& \quad 
			+ e^{C_4 t} \tau \left(\gamma(U^u_{t})\,\beta (U^v_{t}) \right)\Delta_{u,v}[\beta_t] \cdot \Delta_{u,v}[\gamma_t] {\rm d}t  \\
		& \quad 
			+ e^{C_4 t} \tau \left( \Delta_{u,v}[\gamma_t] \cdot \Delta_{u,v} [\beta_t] \right)\beta (U^v_{t})\,\gamma(U^u_{t})  {\rm d}t \\
		& \quad 
			+ e^{C_4 t} \tau \left( \Delta_{u,v}[\gamma_t] \cdot \beta (U^v_{t}) \right)\beta (U^v_{t}) \cdot  \Delta_{u,v}[\gamma_t]{\rm d}t \\			
		& \quad  
			+ e^{C_4 t} \Delta_{u,v} [U_t] \cdot  \Delta_{u,v}[\beta_t] {\rm d} W_t  \gamma(U^u_{t}) +  e^{C_4 t} \Delta_{u,v} [U_t] \cdot \beta (U^v_{t})  {\rm d} W_t  \Delta_{u,v}[\gamma_t]  \\
		& \quad  
			+ e^{C_4 t} \Delta_{u,v}[\beta_t]  {\rm d} W_t  \gamma(U^u_{t}) \Delta_{u,v} [U_t] +  e^{C_4t} \beta (U^v_{t})  {\rm d} W_t  \Delta_{u,v} [\gamma_t] \cdot  \Delta_{u,v} [U_t].
		\end{split}
		\end{equation*}
Integrating both sides of the above equation from $0$ to $t$ and taking the trace, we obtain 
\begin{equation*}
		\begin{split}
		& e^{C_4 t}  \| \Delta_{u,v} [U_t]\|_2^2 - \| \Delta_{u,v} [U_{0}]\|_2^2 \\
		&\quad =2\int_{0}^t e^{C_4 s} \tau\left(\Delta_{u,v} [U_s] \cdot \Delta_{u,v} [\alpha_s] \right) {\rm d} s  +C_4 \int_{0}^t e^{C_4 s} \| \Delta_{u,v} [U_s]\|_2^2 {\rm d}s\\
		& \quad \quad +   \int_{0}^t  e^{C_4 s} \left(\tau \big(\gamma(U^u_{s})\cdot\Delta_{u,v}[\beta_s]\big) \right)^2 {\rm d} s  \\
		&\quad \quad + 2  \int_{0}^t  e^{C_4 s} \tau \left( \Delta_{u,v}[\beta_s]\cdot  \Delta_{u,v}[\gamma_s] \right) \cdot \tau \left(\beta(U^v_{s})\gamma(U^u_{s})\right) {\rm d} s\\
		& \quad \quad +   \int_{0}^t  e^{C_4 s} \left(\tau \big( \Delta_{u,v}[\gamma_s] \cdot \beta (U^v_{s})\big) \right)^2 {\rm d} s,
		\end{split}
		\end{equation*}	
where we note that the terms corresponding to ${\rm d} W_t$ all vanish by Lemma~\ref{lem:tau0}. Indeed, each such term is a finite linear combination of terms of the form $\tau\left(\int_0^t A_s\, {\rm d} W_s\,B_s\right),$ where it is enough to consider $A_s=e^{C_4s}U_s^x\beta(U_s^y),$ $B_s=\gamma(U_s^z)$ or $A_s=e^{C_4s}\beta(U_s^x),~B_s=\gamma(U_s^y)U_s^z$
with $x,y,z\in\{u,v\}$. These processes are $\mathscr{A}_{\le s}$-adapted. Moreover, using the boundness condition \eqref{eq:bound-s} on $\beta$ and $\gamma$, we have, for every fixed $t\ge 0$,
\begin{equation*}
\int_0^t 
\|e^{C_4s}U_s^x\beta(U_s^y)\|_2^2\cdot
\|\gamma(U_s^z)\|_2^2\,{\rm d}s 
\le
e^{2C_4t}K_\beta^2K_\gamma^2
\int_0^t
\|U_s^x\|^2\, {\rm d}s
<\infty .
\end{equation*}
Similarly,
\begin{equation*}
\int_0^t 
\|e^{C_4s}\beta(U_s^x)\|_2^2 \cdot
\|\gamma(U_s^y)U_s^z\|_2^2\, {\rm d}s<\infty .
\end{equation*}
Hence Lemma~\ref{lem:tau0} applies, and all stochastic integral terms have trace zero.

Hence, by \eqref{eq:difference-s-1}, \eqref{eq:difference-s-2}, and \eqref{eq:bound-s}, we have 
\begin{equation*}
\begin{split}
 e^{C_4 t} \| \Delta_{u,v} & [U_t]\|_2^2  - \| \Delta_{u,v} [U_{0}]\|_2^2\\
& \leq  (2L'_\alpha+ C_4 )\int_{0}^t e^{C_4 s} \| \Delta_{u,v} [U_s] \|^2_2 {\rm d}s \\
& \quad +L^2_{\beta}\int_{0}^t e^{C_4 s} \left\|\gamma(U^u_{s}) \right\|^2_2 \left\| \Delta_{u,v} [U_s] \right\|_2^2  {\rm d} s\\
& \quad +2L_{\beta}L_{\gamma}\int_{0}^t e^{C_4 s} \left\|\gamma(U^u_{s}) \right\|_2 \left\|\beta(U^v_{s}) \right\|_2  \left\| \Delta_{u,v} [U_s] \right\|_2^2 {\rm d} s\\
& \quad +L^2_{\gamma}\int_{0}^t e^{C_4 s}  \left\|\beta(U^v_{s}) \right\|^2_2  \left\| \Delta_{u,v} [U_s] \right\|_2^2 {\rm d} s\\
& \leq (2 L'_\alpha+C_4+ (L_{\beta}K_{\gamma}+L_{\gamma}K_{\beta})^2) \int_{0}^t e^{C_4 s} \left\|\Delta_{u,v} [U_s]\right\|_2^2 {\rm d}s.
\end{split}
\end{equation*}
Setting $C_4:=-2 L'_\alpha-(L_{\beta}K_{\gamma}+L_{\gamma}K_{\beta})^2> 0,$ we conclude that for any $t \geq 0,$
\begin{equation*}
\| \Delta_{u,v}  [U_t] \|_2^2  \leq e^{-C_4 t}\| \Delta_{u,v} [U_{0}] \|_2^2 =e^{-C_4t}\|u-v\|_2^2.
\end{equation*}
\end{proof}

\begin{proposition}\label{prop:Cauchy-prop-2}
Assume the hypotheses of Theorem~\ref{thm:stationary}. Then there exists a probability measure $\mu$ on $\mathbb{R}$ such that, for any free Brownian motion $\{W_t\}_{t \geq 0}$ and any initial value $u=U_0 \in \mathscr{A}_{sa}$ that is free from $\{W_t\}_{t\geq 0}$, the distribution of the unique solution $\{U^u_t\}_{t \geq 0}$ to \eqref{eq:fSDE} converges to $\mu$ in $W_2$ as $t\to\infty$.
\end{proposition}

\begin{proof}
We adapt the idea of Guionnet and Shlyakhtenko (see \cite[Theorem 2.2]{guionnet_free_2009}). Fix $s\ge 0$, 
let $\{\widehat{W}_t\}_{t \geq 0}$ be a free copy of $\{W_t\}_{t\geq 0}$ that is free from $u.$
By our assumptions, we can construct the process $\{\widehat{U}^u_t\}_t$ as $\{U^u_t\}_t$ on $[0, s],$ namely, 
\begin{equation*}
\widehat{U}^u_t = u +\int_{0}^{t} \alpha(\widehat{U}^u_z) {\rm d} z + \int_{0}^{t} \beta( \widehat{U}^u_z) {\rm d}\widehat{W}_z \gamma (\widehat{U}^u_z)
\end{equation*} 
$t \in [0, s].$ It is clear that $\widehat U_s^u$ is free from $\{W_t\}_{t\ge0}$, and moreover, $\{\widehat{U}_t^u, \widehat{W}_t\}_{t \in [0,s]}$ is a weak solution to \eqref{eq:fSDE} on $[0, s]$. Then, the weak uniqueness of the equation implies that
$\mu_{\widehat U_s^u}=\mu_{U_s^u}.$ 

Let $\widehat U_{t+s}^u$ be the solution starting from $\widehat U_s^u$ and driven by $\{W_t\}_{t\ge0}$. Then Proposition~\ref{prop:Stability-prop-1} implies
\begin{equation*}
\begin{split}
\left\|U_t^u-\widehat U_{t+s}^u\right\|_2 &\le
e^{-ct}\|u-\widehat U_s^u\|_2\\
& \le e^{-ct}\|u-\widehat U_s^u\|
\le Me^{-ct},
\end{split}
\end{equation*}
where $M$ depends on $\|u\|$ and $\sup_{t \in [0, s]} \|\widehat{U}^u_t\|$.

Now consider the weak solution associated with the initial value $\widehat U_s^u$ (resp. $U_s^u$) and the corresponding future free Brownian
motion $\{W_t\}_{t\ge0}$ (resp. $\{W_{t+s}-W_s\}_{t\ge0}$). The weak uniqueness of the equation again implies that
\begin{equation*}
\mu_{\widehat U_{t+s}^u}
=
\mu_{U_{t+s}^u}
\end{equation*}
for any $t \geq 0.$
Therefore,
\begin{equation*}
W_2(\mu_{U_t^u},\mu_{U_{t+s}^u})
=
W_2(\mu_{U_t^u},\mu_{\widehat U_{t+s}^u})
\le
\|U_t^u-\widehat U_{t+s}^u\|_2
\le
Me^{-ct},
\end{equation*}
which implies that $\{\mu_{U_t^u}\}_{t\ge0}$ is Cauchy in $W_2$. Since
$\mathcal P_2(\mathbb R)$ is complete under $W_2$, there exists
$\mu^u\in\mathcal P_2(\mathbb R)$ such that
\begin{equation*}
\mu_{U_t^u}\longrightarrow \mu^u
\quad\text{in }W_2,\quad \text{as } t\to\infty .
\end{equation*}

Moreover, Proposition~\ref{prop:Stability-prop-1} yields that, for any initial values
$u,v\in\mathscr{A}_{sa}$,
\begin{equation*}
W_2(\mu_{U_t^u},\mu_{U_t^v})
\leq \|U_t^u-U_t^v\|_2
\leq e^{-Ct}\|u-v\|_2 \to 0,
\quad \text{as } t\to\infty .
\end{equation*}
Using the triangle inequality, we obtain
\begin{equation*}
W_2(\mu^u,\mu^v)
\leq W_2(\mu^u,\mu_{U_t^u})
+ W_2(\mu_{U_t^u},\mu_{U_t^v})
+ W_2(\mu_{U_t^v},\mu^v).
\end{equation*}
Passing to the limit as $t\to\infty$, the right-hand side vanishes, and hence
$W_2(\mu^u,\mu^v)=0$. Since $W_2$ is a metric, we conclude that $\mu^u=\mu^v$.
Thus the limiting distribution does not depend on the initial value, and we denote
this unique limiting distribution by $\mu$.
\end{proof}

\begin{remark}
In free probability, a free copy of an operator $X \in \mathscr{A}$ is another operator $\widehat{X}$ (usually defined on an enlarged space, for instance, the reduced free product $(\mathscr{A}, \tau) \ast (\mathscr{A}, \tau)$) such that they are freely independent and have the same distribution. Note that in our proof we shift our framework to the enlarged noncommutative probability space after introducing the free copy.  
\end{remark}

\begin{proof}[Proof of Theorem~\ref{thm:stationary}]
We only need to prove that $\mu$ is stationary, where $\mu$ is obtained in Proposition \ref{prop:Cauchy-prop-2}. Let $Y \in \mathscr{A}_{sa}$ with distribution $\mu$ that is free from the driving free Brownian motion $\{W_t\}_{t\ge0}$, and let $\{U^Y_{t}\}_{t \geq 0}$ be the unique solution to the free SDE \eqref{eq:fSDE} with the initial value $Y$.
Since $\mu_{U_s^Y}\to\mu$ in $W_2$ as $s\to\infty$, 
 we can choose $Y_s\in \mathscr A_{sa}$ such that $\mu_{Y_s}=\mu_{U_s^Y}$ and
\begin{equation*}
\|Y_s-Y\|_2
=
W_2(\mu_{U_s^Y},\mu)
\longrightarrow 0\,~\text{as}\,~ s\to\infty.
\end{equation*}

Let $\{\widehat W_t\}_{t\geq0}$ be a free copy of $\{W_t\}_{t\geq 0}$ that is free from $W^*(Y_s, Y)$. Let $\{\widehat U_t^{Y_s}\}_{t\ge0}$ and $\{\widehat U_t^Y\}_{t\ge0}$ be the corresponding solutions to \eqref{eq:fSDE} driven by
$\{\widehat W_t\}_{t\ge0}$. Then Proposition~\ref{prop:Stability-prop-1} gives for any $t\ge 0$
\begin{equation*}
\|\widehat U_t^{Y_s}-\widehat U_t^Y\|_2
\le
e^{-Ct}\|Y_s-Y\|_2\,~\text{as}\,~ s\to\infty.
\end{equation*}
Hence
\begin{equation*}
\mu_{\widehat U_t^{Y_s}}\longrightarrow \mu_{\widehat U_t^Y}\, \text{in}\,~ W_2 \, ~\text{as}\,~ s\to\infty.
\end{equation*}

Now consider the weak solution associated with the initial value $Y_s$ (resp. $U_s^Y$) and the corresponding future free Brownian
motion $\{\widehat W_t\}_{t\ge0}$ (resp. $\{W_{t+s}-W_s\}_{t\ge0}$). Then, the weak uniqueness of the equation implies that 
\begin{equation*}
\mu_{\widehat U_t^{Y_s}}
=
\mu_{U_{t+s}^Y}.
\end{equation*}
Letting $s\to\infty$, we obtain
\begin{equation*}
\mu_{\widehat U_t^{Y_s}}
=
\mu_{U_{t+s}^Y}
\longrightarrow
\mu \,~ \text{in}\,~ W_2.
\end{equation*}
Finally, note that $\{U_t^Y, W_t\}_{t \geq 0}$ and $\{\widehat U_t^Y, \widehat W_t\}_{t \geq 0}$ are weak solutions to the equation. So we conclude that
\begin{equation*}
\mu_{U_t^Y}=\mu_{\widehat U_t^Y}=\mu .
\end{equation*}
Since $t\ge0$ is arbitrary, $\mu$ is a stationary distribution for the free SDE.

It remains to prove the uniqueness of the stationary distribution. 
Let $\nu$ be another stationary distribution, and let $Z$ be an initial value such that
$\mu_Z=\nu$. Since $\nu$ is stationary, we have
\begin{equation*}
\mu_{U_t^Z}=\nu, \qquad t\geq 0.
\end{equation*}
By Proposition~\ref{prop:Cauchy-prop-2}, however, $\mu_{U_t^Z}$ converges to $\mu$
in $W_2$ as $t\to\infty$. Therefore,
\begin{equation*}
W_2(\nu,\mu)
=
\lim_{t\to\infty} W_2(\mu_{U_t^Z},\mu)
=0.
\end{equation*}
Since $W_2$ is a metric, it follows that $\nu=\mu$. Thus the stationary distribution
is unique.
\end{proof}

 
 \subsection{Exponentially decay of the stationary distribution}
 
Let $f: \mathbb{R} \to \mathbb{C}$ be a $C^{k+2}$ function with compact support. Then its Fourier transform is well-defined. Let $\hat{f} (y) = \int_{\mathbb{R}} f(x) e^{-iyx} {\rm d}x.$ It was known \cite{korner1988fourier} that
there exists a constant $C>0$ such that 
\begin{equation*}
|y|^k |\hat{f}(y)|  \leq \frac{C}{1+ y^2}.
\end{equation*}
Therefore, we have 
\begin{equation*}
\int_{\mathbb{R}} |y|^k |\hat{f}(y)| {\rm d}y < \infty.  
\end{equation*}

On the other hand, by Duhamel's formula we have (see also \cite{Azamov2009,jekel2022tracial})
\begin{equation*}
f(U) - f(V)= \int_0^1 \int_{\mathbb{R}} y \,
e^{i s y U}(U - V)e^{i(1-s)y V} \,
\hat{f}(y) \, {\rm d} y\,{\rm d}s 
\end{equation*}
for any $U, V \in \mathscr{A}_{sa}.$

\begin{proof}[Proof of Corollary~\ref{cor:exp-conv}]
Let $u$ and $v$ be self-adjoint operators with the same distribution $\mu$, 
and let $\{U^u_t\}_{t \geq 0}$ and $\{U^v_t\}_{t \geq 0}$ be the unique solutions to \eqref{eq:fSDE} 
with initial values $u$ and $v$, driven by free Brownian motions that are free from $u$ and $v$, respectively.
Now let $\{\widehat{W}_t\}_{t \geq 0}$ be a free Brownian motion that is free from $W^*(u, v)$, and denote by $\{\widehat{U}^u_t\}_{t \geq 0}$ and $\{\widehat{U}^v_t\}_{t \geq 0}$ 
the corresponding solutions driven by this common $\{\widehat{W}_t\}_{t \geq 0}$.
Then Proposition~\ref{prop:Stability-prop-1} yields, for any $t \geq 0$,
\begin{equation*}
    \|\widehat{U}^u_t - \widehat{U}^v_t\|_2 
    \leq e^{-C_4 t} \|u - v\|_2 
    \leq 2 e^{-C_4 t} \|u\|_2.
\end{equation*}
Since $U^v_t$ has the stationary distribution $\mu$ and the free Brownian motions share the same distribution, 
it follows that for any $t \geq 0$ and any $f \in C_c^3(\mathbb{R})$,
\begin{equation*}
\begin{split}
\bigl|\tau\bigl(f(U^u_t)\bigr)& -\int_{\mathbb{R}}f(x)\, \mu( {\rm d} x)\bigr|\\
& = \bigl|\tau\bigl(f(U^u_t)\bigr)-\tau\bigl(f(U_t^v)\bigr)\bigr| = \bigl|\tau\bigl(f(\widehat U^u_t)\bigr)-\tau\bigl(f(\widehat U_t^v)\bigr)\bigr|\\
& =\bigl|\tau \int_0^1 \int_{\mathbb{R}} y \,
e^{i s y \widehat U^u_t}(\widehat U^u_t - \widehat U_t^v)e^{i(1-s)y \widehat U_t^v} \,
\hat{f}(y) \, {\rm d} y\, {\rm d} s   \bigr|\\
& \leq \int_0^1 \int_{\mathbb{R}} |y| \,
\bigl\| e^{is y \widehat U^u_t}(\widehat U^u_t - \widehat U_t^v)e^{i(1-s)y \widehat U_t^v}\bigr\|_2 \,
|\hat{f}(y)|\, {\rm d} y\, {\rm d}s \\
&= \int_{\mathbb{R}} |y\hat{f}(y)|\,{\rm d} y \cdot \| \widehat U^u_t - \widehat U_t^v \|_2.
\end{split}
\end{equation*}

It follows that
\begin{equation*}
\bigl|\tau\bigl(f(U^u_t)\bigr)-\int_{\mathbb{R}}f(x)\, \mu( {\rm d} x)\bigr|
\le C_f\, e^{-C_4 t}\|u\|_2,
\end{equation*}
with $C_f := 2 \int_{\mathbb{R}}|y\hat{f}(y)|\, {\rm d} y$. This completes the proof.
\end{proof}
 
\section{Examples}\label{sec:exa}

In this section, we provide some concrete examples to illustrate the applicability of our main results.

\subsubsection*{Free Diffusions}

Recall that Biane and Speicher \cite{bianeFreeDiffusionsFree2001} introduced free diffusions of the form
\begin{equation*}
d U_t = \alpha(U_t) {\rm d}t + {\rm d}W_t.
\end{equation*}
They established the existence and uniqueness of the solution under the following conditions \cite[Theorem 3.1]{bianeFreeDiffusionsFree2001}: the drift $\alpha$ is locally operator Lipschitz, and moreover, for any self-adjoint $U,$
\begin{equation}\label{eq:biane}
U\alpha(U)+\alpha(U)U+{\bf 1} \le a U^2 + b,
\end{equation}
for some constants $a\in\mathbb{R},\,b\ge0.$ Moreover, they also obtained the stationary case of the free diffusions in \cite[Section 4]{bianeFreeDiffusionsFree2001}.

In the present work, we replace the locally operator Lipschitz condition with the condition stated in Assumption~\ref{assumption:ass1}. Moreover, condition \eqref{eq:biane} is relaxed to the more general condition \eqref{eq:ohtj} in Assumption~\ref{assumption:ass2}. 
More precisely, consider drifts of the form
\begin{equation}
\alpha(U)=c\,U+f(U),
\end{equation}
where $f$ is defined via function calculus from a Lipschitz function $f:\mathbb{R}\to\mathbb{R}$ with constant $L_f$ and satisfying $f(0)=0.$ Note that every Lipschitz function is operator Lipschitz in the $L^2$-norm (see \cite[Theorem 2]{potapov_operator-lipschitz_2011}). So the drift $\alpha$ is operator Lipschitz on $L^2(\mathscr{A},\tau).$  
Furthermore,
\begin{equation*}
\begin{split}
2\tau\bigl(U\alpha(U)\bigr) + 1
&= 2c\|U\|_2^2 + 2\tau\bigl(Uf(U)\bigr) + 1 \\
&\le 2(c+L_f)\|U\|_2^2  + 1, 
\end{split}
\end{equation*}
so that condition \eqref{eq:ohtj} holds with $L_1=2c+2L_f$ and $L_2=1$. According to Theorem~\ref{thm:solution}, a unique global solution exists for every $c\in\mathbb{R}$. 

A crucial observation is that not every (scalar) Lipschitz function is operator Lipschitz; the canonical counter-example is $f(x)=|x|$ (Lipschitz constant $1$), which fails to be operator Lipschitz. Hence, the choice
\begin{equation}
\alpha(U)=c\,U+|U|
\end{equation}
provides an explicit drift that satisfies our Assumptions~\ref{assumption:ass1} and~\ref{assumption:ass2} (for suitable $c$) but does not satisfy the (locally) operator Lipschitz condition required by Biane and Speicher. This concrete example underscores the intrinsic difference between the two settings.

For the stationary distribution of the free diffusion, we consider 
\begin{equation*}
\alpha(U)=-\frac{1}{2} U-\frac{g}{2} U^3
\end{equation*}
where $g$ is a negative constant but sufficiently close to zero. It is clear that $\alpha$ is locally operator Lipschitz and satisfies the condition \eqref{eq:biane}. Hence, there exists a solution defined for all $t\ge 0$, which moreover remains uniformly bounded in operator norm. Weak uniqueness follows from Remark~\ref{rmk:fdiff-weak}. Using the non-commutative Cauchy-Schwarz inequality, we obtain
\begin{equation*}
\begin{split}
 -\frac{g}{2} \langle U-V,\;U^3-V^3\rangle 
&= -\frac{g}{2} \langle U-V,\;(U-V)U^2+VU(U-V)+V(U-V)V\rangle\\
&=  -\frac{g}{2} (\tau((U-V)^2U^2)+ \tau(VU(U-V)^2)+ \tau((V(U-V))^2))\\
&\le  -\frac{g}{2} (\|U-V\|_2^2\|U\|^2+ \|V\|\|U\|\|U-V\|_2^2+ \|V\|^2\|U-V\|_2^2)\\
&=  -\frac{g}{2}(\|U\|^2+\|U\|\|V\|+\|V\|^2)\|U-V\|_2^2.
\end{split}
\end{equation*}
Denote the uniformly boundedness of solution by $M.$ Then
\begin{equation*}
\begin{split}
\langle U-V,\;\alpha(U)-\alpha(V)\rangle
&= -\frac{1}{2}\|U-V\|_2^2 -\frac{g}{2} \langle U-V,\;U^3-V^3\rangle \\
&\le (-\frac{1}{2}-\frac{3g}{2}M^2)\|U-V\|_2^2,
\end{split}
\end{equation*}
which is precisely condition \eqref{eq:alpha-onesideLip-loc} with $L'_\alpha=-\frac{1}{2}-\frac{3g}{2}M^2$. Hence, for $0> g>-\frac{1}{3M^2},$ Theorem~\ref{thm:stationary}  guarantees the existence of a unique stationary distribution, which gives the affirmative answer to the conjecture in \cite{bianeFreeDiffusionsFree2001}.


\subsubsection*{Free geometric Brownian motion}

A symmetric version of the free geometric Brownian motion, studied in \cite{karginFreeStochasticDifferential2011}, is described by
\begin{equation}\label{eq:gbm2}
d U_t = a U_t {\rm d}t + U_t {\rm d}W_t + {\rm d}W_t U_t.
\end{equation}
It is clear that Assumption~\hyperlink{assumption:ass3}{(B-1)} holds trivially with $L'_\alpha=a,L_{\beta^1}=1,L_{\gamma^1}=0,L_{\beta^2}=0,L_{\gamma^2}=1.$ 
By Proposition~A.1 in \cite{capitaine_free_2005}, the free geometric Brownian motion has a unique solution in \((\mathscr A,\tau)\), and this solution is uniformly bounded in the operator norm.  Weak uniqueness is obtained by applying Remark~\ref{rmk:weak-unique}. Therefore, Theorem~\ref{thm:stationary} ensures that the equation admits a unique stationary distribution provided that \(a<-2\).

\subsubsection*{An example with a non-linear drift}
The previous examples have linear drift coefficient. Now consider 
\begin{equation}\label{eq:gbm2}
d U_t = (a U_t+b(\sin U_t)^3) {\rm d}t + c U_t {\rm d}W_t + c {\rm d}W_t U_t,
\end{equation}
for some constants $a, b, c \in \R.$ 
Note that $(\sin x)^3$ is operator Lipschitz. By Proposition~A.1 in \cite{capitaine_free_2005}, the equation~\eqref{eq:gbm2} admits a unique solution in \((\mathscr A,\tau)\), and this solution remains uniformly bounded in the operator norm. Remark~\ref{rmk:weak-unique} yields the desired weak uniqueness.

Then for any $U, V \in  L^2(\mathscr{A},\tau)_{sa},$
\begin{equation*}
\begin{split}
&\langle U-V,\alpha(U)-\alpha(V)\rangle= a\|U-V\|^2_2 + b\tau(((\sin U)^3-(\sin V)^3)(U-V))\\
&\leq a\|U-V\|^2_2 + b\tau\big((\sin U)^2(\sin U-\sin V)(U-V)\\
&\quad+\sin U(\sin U-\sin V)\sin V(U-V)+(\sin U-\sin V)(\sin V)^2(U-V)\big)\\
&\leq a\|U-V\|^2_2 + |b|\big(\|\sin U\|^2\|\sin U-\sin V\|_2\|U-V\|_2+\|\sin U\|\\
&\quad\cdot\|\sin U-\sin V\|_2\|\sin V\| \|U-V\|_2+\|\sin U-\sin V\|_2\|\sin V\|^2\|U-V\|_2\big)\\
&\leq (a+3|b|)\|U-V\|^2_2,
\end{split}
\end{equation*}
where we used $\|\sin U\|\leq1.$ So Assumption~\hyperlink{assumption:ass3}{(B-1)} holds with $L'_\alpha=a+ 3|b|,L_{\beta^1}=|c|,L_{\gamma^1}=0,L_{\beta^2}=0,L_{\gamma^2}=|c|.$ 
According to Theorem~\ref{thm:stationary}, the equation \eqref{eq:gbm2} admits a unique stationary distribution if $a <-3|b|-2c^2.$
\bigskip

{\bf Acknowledgment.} 
We thank Prof. Roland Speicher for valuable discussions. We are partially supported by NSFC No. 12571147 and Provincial Natural Science Foundation of Hunan (grand number 2024JJ1010). J. Wei further acknowledges the support of Saarland University, which provided a fruitful environment to work on this project.

\bibliographystyle{acm}
\bibliography{ref}

\end{document}